\documentclass[reqno]{amsart}
\usepackage{latexsym,amscd,amssymb,amsopn,amsthm,amsfonts,amsmath}
\usepackage{color}
\copyrightinfo{2010}{American Mathematical Society}

\newtheorem{theorem}{Theorem}[section]
\newtheorem{proposition}[theorem]{Proposition}
\newtheorem{lemma}[theorem]{Lemma}
\newtheorem{corollary}[theorem]{Corollary}

\theoremstyle{definition}

\newtheorem{example}[theorem]{Example}

\theoremstyle{remark}
\newtheorem{remark}[theorem]{Remark}

\numberwithin{equation}{section}

\begin{document}

\title[Existence theorems for ersatz Navier-Stokes type equations]
      {Existence theorems for  
       regular \\ spatially periodic solutions to \\
       ersatz Navier-Stokes equations
			} 
			
\author{A. Shlapunov}

\address[Alexander Shlapunov]
        {Siberian Federal University, 
         Institute of Mathematics and Computer Science, 
         pr. Svobodnyi, 79, 660041 Krasnoyarsk, Russia
}

\email{ashlapunov@sfu-kras.ru}

\subjclass [2010] {Primary 76N10; Secondary 35Q30, 76D05}

\keywords{Navier-Stokes type equations,
          smooth solutions,
          existence theorem}

\begin{abstract}
The initial problem for the Navier-Stokes type equations over
${\mathbb R}^n \times [0,T]$, $n\geq 2$, with a positive  time $T$ in the spatially periodic 
setting is considered. First, we 
prove that the problem induces an open injective continuous 
mapping on scales of specially constructed function spaces of Bo\-chner-Sobolev
type over the $n\,$-dimensional torus ${\mathbb T}^n$. Next,  
rejecting the idea of proving a universal a 
priori estimate for high-order derivatives, we obtain a surjectivity 
criterion for the non-linear mapping under the considerations 
in terms of boundedness for its 
inverse images of precompact sets. Finally, 
we prove that the mapping is  surjective 
if we consider the versions of the Navier-Stokes type 
equations containing no `pressure'{}. 
This gives a uniqueness and existence theorem for regular solutions to this 
particular ersatz of the Navier-Stokes type equations. 
The  used techniques  consist in proving the closedness of the image by estimating all possible divergent sequences in the preimage and matching the asymptotics. 
The  following facts are essential:
i) the torus is a compact closed manifold, ii) the corresponding system is `local'. 
\end{abstract}

\maketitle

\section*{Introduction}
\label{s.0}

The problem\footnote{For $n=3$ some results of this paper 
obtained in the collaboration with  Prof. 
Nikolai Tarkhanov (Instutut f\"ur Mathematik, Universit\"at Potsdam),  
who  passed away in  2020, see \cite{ShlTaArxiv}. }
of describing the dynamics of incompressible viscous fluid is of great importance 
in applications. The dynamics is described by the Navier-Stokes equations and the problem 
consists in finding a sufficiently regular solution to the equations for which a uniqueness 
theorem is available,  cf. \cite{Lady03}. Essential contributions has been published in 
the research articles   \cite{Lera34a,Lera34b},
   \cite{Kolm42},
   \cite{Hopf51},
as well as surveys and books
   \cite{Lady70}),
   \cite{Lion61,Lion69},
   \cite{Tema79},
   \cite{FursVish80},
etc.

After Leray \cite{Lera34a,Lera34b}, a great attention was paid to weak solutions 
to the Navier-Stokes equations 
in cylindrical domains in ${\mathbb R}^3 \times [0,+\infty)$.
E. Hopf \cite{Hopf51} proved the existence of weak solutions to \eqref{eq.NS} satisfying 
reasonable estimates. However, in this full generality no uniqueness theorem for a weak 
solution has been known. On the other hand, under stronger conditions on the solution, it is 
unique, cf. \cite{Lady70,Lady03} by O.A. Ladyzhenskaya who proved the 
existence of smooth solutions for the 
two-dimensional version of problem \eqref{eq.NS}. Some authors (see, e.g., Leray 
\cite{Lera34a,Lera34b} and also a recent paper \cite{Tao16} by T. Tao) 
expressed a lot of skepticism on the existence of regular solutions for all regular data in 
${\mathbb R}^3$. 

We  consider  an initial 
problem for a  slightly more general Navier-Stokes type equations. 
Let $\Delta  = \partial^2_{x_1} + \partial^2_{x_2} + \dots  \partial^2_{x_n}$ be the Laplace 
operator,   $\nabla$ and $\mathrm{div}$ be the gradient operator and the divergence operator, 
respectively, in the Euclidean space ${\mathbb R}^n$, $n\geq 2$. 
Fixing a  zero order bilinear matrix differential operator 
$
M: {\mathbb R}^n \times {\mathbb R}^{n^2} \to 
{\mathbb R}^n 
$ 
with constant coefficients we set for vector fields $u,w$
\begin{equation}\label{eq.non-linearity}
{\mathcal M} (u,w) = M (u,\nabla w) , 
\, {\mathbf D} u = {\mathcal M} (u,u). 
\end{equation}

In the sequel we consider the following initial problem.
Given any sufficiently regular vector-valued functions
   $f = (f^1, f^2, \dots f^n)$ and
   $u_0 = (u_{0}^1, u_{0}^2, \dots u_{0}^n)$
on
   ${\mathbb R}^n \times [0,T]$ and
   ${\mathbb R}^n$,
respectively, find a pair $(u,p)$ of sufficiently regular functions
   $u = (u^1, u^2, \dots u^n)$ and
   $p$
on ${\mathbb R}^n \times [0,T]$ satisfying
\begin{equation}
\label{eq.NS}
\left\{
\begin{array}{rcll}
   \partial _t u     - \mu \Delta u + 
{\mathbf D} u + a \, \nabla p
 & =
 & f,
 & (x,t) \in {\mathbb R}^n \times (0,T),
\\
a\,    \mbox{div}\, u
 & =
 & 0,
 & (x,t) \in {\mathbb R}^n \times (0,T),
\\[.05cm]
   u
& =
& u_0,
& (x,t) \in \mathbb{R}^n \times \{ 0 \}
\end{array}
\right.
\end{equation}
with fixed real numbers $T>0$, $\mu>0$  and a parameter 
$a$ that can take the values $a=0$ and $a=1$.  
We additionally assume that the data $f$ and $u_0$ are spatially periodic with a period 
$\ell > 0$, i.e.,  for any $1 \leq j \leq n$ we have
\begin{equation*}
   f (x+\ell e_j,t)
  = f (x,t),
\,\, 
   u_0 (x+\ell e_j)
  =  u_0 (x)
\end{equation*}
whenever $x \in {\mathbb R}^n$ and $t \in [0,T]$,
    where $e_j$ is as usual the $j\,$-th unit basis vector in ${\mathbb R}^n$.
Then, the solution $(u,p)$ is also looked for in the space of spatially periodic functions 
with period $\ell$ on ${\mathbb R}^n \times [0,T]$.

\begin{example} \label{ex.NS.intr}
If we choose $a=1$ and ${\mathbf D} (u) =  (u \cdot \nabla) u $
then relations \eqref{eq.NS} are usually referred to as but 
the Navier-Stokes equations for incompressible fluid  with given
   dynamical viscosity 	$\mu$ of the fluid, 
   density vector of outer forces $f$, 
   the initial velocity $u_0$
and the search-for velocity vector field $u$ and the pressure $p$ of the flow, see for instance
   \cite{Tema79}
for the classical setting or
   \cite{Serr59b},
   \cite{Tema95}
for the periodic setting.
\end{example}

\begin{example} \label{ex.SvPl.intr}
If $\mu=1$, $a=0$, $b\in (0,1)$ is a real parameter and 
\begin{equation} \label{eq.SvPl}
{\mathbf D} u
 = b \, (u \cdot \nabla) u + 
\frac{1}{2} (1-b) \, \nabla |u|^2 + 
\frac{1}{2} (\mathrm{div } u) u 
\end{equation}
 then \eqref{eq.NS} becomes the   
non-linear problem in ${\mathbb R}^n \times [0,T)$ considered in paper
\cite{PlSv03} by P. Plech\'a$\rm \check{c}$  and 
V. $\rm \check{S}$ver\'ak  without the 
periodicity assumptions on data and solutions. Actually, 
paper \cite{PlSv03}  provides an instructive  example of a non-linear problem in 
${\mathbb R}^n \times [0,T)$, structurally similar to the Cauchy problem for the 
Navier-Stokes equations and `having the same energy estimate'{}, but,  according to some 
motivations including  numerical simulations, admitting singular 
solutions of special type for smooth data if $n\geq 5$. 
Namely, they consider  the  `radial vector fields' 
\begin{equation} \label{eq.radial.u}
u= - v(r,t) x,
\end{equation}
with functions $v$ of variables $t$ and $r=|x|$. Under the hypothesis of this example 
the fields are solutions to 
\eqref{eq.NS} for $f=0$ and $u_0 = - v(r,0) x$ if 
\begin{equation} \label{eq.radial.eq}
v'_t = v''_{rr} + \frac{n+1}{r} v'_r +  (n+2) v^2 +3r v v' _r . 
\end{equation}
Next,  for $v$ satisfying \eqref{eq.radial.eq} they consider  the self-similar solutions 
\begin{equation} \label{eq.self-sim.v}
v (r,t) = \frac{1}{2\varkappa (T-t)} w \Big( \frac{r}{\sqrt{2\varkappa (T-t)}} \Big)
\end{equation}
with  functions $w (y)$ binded by the following relations, see \cite[(1.9)-(1.11)]{PlSv03}: 
\begin{equation} \label{eq.self-sim}
w'' +  \frac{n+1}{y} w' -  \varkappa \, y \, w' + 
(n+2) w^2 +3y w w'  
 - \textcolor{red}{2} \varkappa \, w =0, \, y\in (0,+\infty),
\end{equation}
\begin{equation} \label{eq.self-sim.Cauchy}
w(0) = \gamma\geq 0, \, w'(0) = 0 ,\, 
w(y) = y^{-2} \mbox{ as } y \to + \infty, 
\end{equation}
with a positive parameter $\varkappa$ (here, in comparison with \cite[(1.9)]{PlSv03}, the 
missed multiplier \textcolor{red}{$2$} in the last term of \eqref{eq.self-sim} is recovered).
Based on some analysis of solutions to the steady equation related to 
\eqref{eq.radial.eq} and  numerical simulations, they made conclusion that  
for $n>4$ self-similar solutions  \eqref{eq.self-sim.v}  may produce singular 
solutions in finite time to  this particu\-lar version of \eqref{eq.NS} 
for regular data via formula \eqref{eq.radial.u} if $\gamma>0$. 
We do not know  to what extent our recovery of the missing multiplier affects the analysis
and the numerical simulation made in \cite{PlSv03};
however it might be, the numerical simulations can not be arguments in analysis.  
Actually, taking into account the type of the  obtained below 
implicit estimate \eqref{eq.Gron.large} for solutions to \eqref{eq.NS} with non-linearity 
\eqref{eq.SvPl} and Remark \ref{r.implicit}, questions arise on  the credibility of the 
numerical simulations involving regular/smooth solutions to these  equations. 
In any case,  the  solutions of the discussed type  do not 
fit the periodic pattern of our paper. Indeed, 
the periodicity assumption on the related  data means that 
$$
(x+\ell e_j)w (|x+\ell e_j|) = x w(|x|) 
\mbox{ for all } x \in {\mathbb R}^n,  1\leq j\leq n,
$$  
and therefore 
$
w (|x+\ell e_j|) = w(|x|) =0 
$
for all $x \in {\mathbb R}^n$ and $1\leq j\leq n$. We note also that there are 
obstacles for the existence of blow-up non-periodic solutions to \eqref{eq.NS}  
with smooth data vanishing on the boundary of a bounded domain in ${\mathbb R}^n$  
(in particular, with smooth compactly supported data),  
related to the radial vector fields. 
This follows from  \cite[Lemma 2.1]{PlSv03} because according to it 
solutions to \eqref{eq.self-sim}, \eqref{eq.self-sim.Cauchy} 
are positive on $(0,+\infty)$ if $\gamma>0$. 
\end{example} 

Of course, for $n=2$ initial problem \eqref{eq.NS} can be successfully treated 
with classical energy type estimates, see, for instance, \cite{Tema95}. 
The aim of this paper is to shade some light on  existence and uniqueness 
theorems for smooth solutions related to smooth data for 
equations \eqref{eq.NS} in the spatially periodic setting  in ${\mathbb R}^n\times (0,T)$, 
$n\geq 3$, corresponding to different values of the parameter $a$ and 
a non-linearity \eqref{eq.non-linearity} 
including \eqref{eq.SvPl}. Instead of using technique 
of weak solutions to non-linear equations, 
we will employ proper linearisations and  the Implicit Function Theorem for Banach spaces.
Our approach is based on the following rather standard observations.
  
Due to  \cite{Pro59},
   \cite{Serr62},
   \cite{Lady70} and
   \cite{Lion61,Lion69},
it is known that  the uniqueness  and improvement of regularity 
for a weak solution to the classical Navier-Stokes equations actually follow if 
the solutions belong to the Bochner class 
$L^\mathfrak{s} ([0,T], L^\mathfrak{r} ({\mathbb R^n}))$  with
\begin{equation}\label{eq.s.r} 
2/\mathfrak{s} + n/\mathfrak{r} = 1 \mbox{ and }   \mathfrak{s} \geq 2, 
  \mathfrak{r} > n
	\end{equation}
(the limit case $  \mathfrak{r}  =n= 3$ was added to the list in \cite{ESS03}).

Next, we note that the existence of regular solutions  to the classical 
Navier-Stokes equations for sufficiently small data in different 
spaces  is known since J.~Leray. In addition to these results, 
O.A.~Ladyzhenskaya discovered  the so-called stability property for the Navier-Stokes 
equations in some Bochner type spaces (see 
\cite[Ch.~4, \S~4, Theorems 10 and 11]{Lady70}). Namely, if for sufficiently regular data 
$(f,u_0)$  there is a sufficiently regular solution $(u,p)$ to the Navier-Stokes equations, 
then there is a neighbourhood of the data in which all the elements admit solutions with the 
same regularity, cf. also \cite{ShlTa21} for Bochner-Sobolev type 
spaces over ${\mathbb R}^n \times (0,T)$ without periodicity assumptions or \cite{Po22} 
for the subject in the context of general elliptic differential complexes.

Thus, according to these observations, 

1) We treat the Navier-Stokes type equations within the framework of the theory of operator 
equations associating them with a continuous \textit{non-linear} mapping  between Banach 
spaces, avoiding weak solutions.  
More precisely, we define two scales $\{B^{s}_1\}_{s \in {\mathbb Z}_+}$,   
$\{B^{s}_2\}_{s \in {\mathbb Z}_+}$  of separable Banach spaces such that 
a) each space of the scale $\{B^{s}_1\}_{s \in {\mathbb Z}_+}$ is continuously 
embedded into the spaces $L^{{\mathfrak s} } ([0,T], L^{{\mathfrak r} } ({\mathbb T}^n))$
for $n\,$-dimensional torus ${\mathbb T}^n$ with any ${\mathfrak s} $, 
${\mathfrak r}$,  satisfying \eqref{eq.s.r}; 
b)  the Navier-Stokes type equations induce non-linear continuous mappings 
${\mathcal A}^s:  B^{s}_1 \to B^{s-1}_2$ for all $s \in {\mathbb N}$; 
c) the components of vector fields belonging to the intersections 
$\cap_{s=1}^\infty B^{s}_1$, $\cap_{s=0}^\infty B^{s}_2$ are 
infinitely differentiable functions on the torus. 

2) 
We use in full the mentioned above stability property for the Navier-Stokes type 
equations discovered by O.~A. Ladyzhenskaya.  Using the Implicit Function Theorem for 
Banach space and the solvability of  the linearised problem related to the Fr\'echet 
derivative of the mappings generated by \eqref{eq.NS}, cf. similar linear problems 
in   \cite{Lady67} or \cite[Ch.~3, \S~1--\S~4]{LadSoUr67}, we extend this property 
to the mappings ${\mathcal A}^s:  B^{s}_1 \to B^{s-1}_2$ with arbitrary 
$s \in {\mathbb N}$, expressing it as open mapping theorem for Navier-Stokes type 
equations \eqref{eq.NS} with a non-linearity of the form \eqref{eq.non-linearity}, 
see Theorem \ref{t.open.NS.short}. 
Then the standard topological arguments immediately imply that a 
non-empty open connected set in a topological vector space coincides 
with the space itself if and only if the set is closed. 

3) In order to prove that the set of data admitting regular 
solutions to the Navier-Stokes type Equations is closed, we do  
not use the Faedo-Galerkin formal 
series replacing them by real approximate solutions to the equations. 
It appears that in the chosen function spaces the closedness of the image is 
equivalent to the boundedness of the sequences in the preimage correspon\-ding to sequences 
converging to an element of the image's closure. Using the integrability as a tool, we 
show that the  above boundedness property can be reduced to  an 
$L^\mathfrak{s} ([0,T], L^\mathfrak{r} ({\mathbb T}^n))$-energy type estimates 
for inverse images of precompact sets under the map ${\mathcal A}_s:  B^{s}_1
 \to B^{s-1}_2$ induced by  \eqref{eq.NS}, see Theorem \ref{t.sr}. 
This echoes the idea of using the properness property
to study nonlinear operator equations, see for instance \cite{Sm65}. 

4) Finally, rejecting the idea of proving a universal a 
priori estimate we obtain an \textit{implicit} exponential  energy type 
estimates for the inverse images of precompact sets 
 in the case where $a= 0$; of course, an $L^\mathfrak{s} ([0,T], L^\mathfrak{r} 
({\mathbb T}^n))$-estimate follows from the exponential one, see Theorem 
\ref{t.exist.Bohcher}. 
Here by a universal estimate we mean a bound for the solutions $(u,p)$ via the data 
$(f,u_0)$ and their derivatives with the functional expressions independent on the entries. 
The  used techniques  consist in proving the closedness of the image by 
estimating all possible divergent sequences in the preimage and matching the asymptotics. 
This gives a uniqueness and existence theorem for regular solutions to this 
particular ersatz of the Navier-Stokes type equations, with non-linearities of type 
\eqref{eq.non-linearity} including \eqref{eq.SvPl}.  
The following facts are essential:
i) the torus is a compact closed manifold, ii) the corresponding system is `local'. 
If  $a=1$ then the arguments fail,  because in this case we have to estimate the 
pseudo-differential term $\mathbf{P} \mathbf{D} u$ with the Leray-Helmholtz projection 
$\mathbf{P} $ instead of the \textit{local} term $\mathbf{D} u$ and the integral operator 
$\mathbf{P}$ does not admit point-wise estimates. 

\section{Preliminaries}
\label{s.prelim}

As usual, we denote by $\mathbb{Z}_{+}$ the set of all nonnegative integers including zero, 
and by $\mathbb{R}^n$ the Euclidean space of dimension $n \geq 2$ with coordinates 
$x = (x^1, \ldots, x^n)$.

The  (discrete) Young inequality will be of frequent use in this paper:
 given any $N = 1, 2, \ldots$, for all positive numbers $a_j$ 
and all numbers $p_j \geq 1$ satisfying $\sum_{j=1}^N 1/p_j = 1$ we have 
\begin{equation}
\label{eq.Young}
   \prod_{j=1}^N a_j \leq \sum_{j=1}^N \frac{a_j^{p_j}}{p_j}.
\end{equation}
	
In the sequel we use systematically the Gronwall lemma  and its generalizations.

\begin{lemma}
\label{l.Perov}
Let   $\mathfrak{A} \geq 0$ be constant and let $\mathfrak{B}$,  $\mathfrak{C}$,  
$\mathfrak{F}$ be nonnegative continuous functions defined on a segment $[a_0,b_0]$ 
and let $\mathfrak{F}$ satisfies the integral inequality
\begin{equation*} 
   \mathfrak{F} (t) \leq \mathfrak{A}  + \int_{a_0}^t (\mathfrak{B} (s) \mathfrak{F} (s) + 
	\mathfrak{C} (s) (\mathfrak{F} (s))^{\gamma_0}) ds
\end{equation*}
for all $t \in [a_0,b_0]$. 
If $0 < \gamma_0 < 1$ then for all $t \in [a_0,b_0]$ we have 
\begin{equation*} 
   \mathfrak{F} (t)
 \leq
   \left( \mathfrak{A}^{1-\gamma_0} e^{ \left( (1-\gamma_0) \int_{a_0}^t  \mathfrak{B} 
	(\tau) d\tau \right)} 
        + (1-\gamma_0) \int_{a_0}^t \mathfrak{C} (\tau) e^{ \left( (1-\gamma_0) \int_\tau^t 
				\mathfrak{B} (s) d s \right) 
				} d\tau  \right)^{\frac{1}{1-\gamma_0}}.
\end{equation*}
If  $\gamma_0 =1$ then for all $t \in [a_0,b_0]$ we have 
\begin{equation*} 
   \mathfrak{F} (t)
 \leq
   \mathfrak{A} e^{ \left( \int_{a_0}^t ( \mathfrak{B} (s) + \mathfrak{C}(s))ds  \right)}.
\end{equation*}
Moreover, if $\gamma_0 >1$ and there is $h\in (0,b_0-a_0]$ such that 
\begin{equation} \label{eq.Gron.add} 
\mathfrak{A} \left((\gamma_0-1) \int_{a_0}^{a_0+h} \mathfrak{C}(\tau) d \tau \right)^
{\frac{1}{\gamma_0-1}}
<\left(e^{\left(\int_{a_0}^{a_0+h} 
(1-\gamma_0)\mathfrak{B}(\tau) d \tau \right)}\right)^{\frac{1}{\gamma_0-1}}
\end{equation}
then for all $t \in [a_0,a_0+h]$ we have
\begin{equation*} 
\mathfrak{F} (t) \leq \mathfrak{A} \left( 
e^{\left(\int_{a_0}^t (1-\gamma_0)\mathfrak{B}(\tau) d \tau \right)} -
\textcolor{red}{\mathfrak{A}^{\gamma_0-1}}(\gamma_0-1) \int_{a_0}^t \mathfrak{C}(\tau) 
e^{\left(\int_\tau^t (1-\gamma_0)\mathfrak{B}(s) d s \right) 
 } d \tau \right)^{\frac{1}{\textcolor{red}{1-\gamma_0}}}.
\end{equation*}
\end{lemma}

\begin{proof}
For $\gamma_0=1$ this gives the Gronwall lemma,  see, for instance, 
\cite[Ch. XII, p. 353, formula (1.2)]{MPF91}. 
For $0 < \gamma_0 \ne 1$ it was proved 
by A.I.~Perov \cite{Per59}, see also \cite[Ch. XII, p.~360, 
formulas (9.1$'$)--(9.1$'''$)]{MPF91} 
(for $\gamma_0>1$ in the  last book two 
obvious misprints are corrected in formula (9.1$'''$)). 
\end{proof}

For a measurable set $\sigma$ in ${\mathbb R}^n$ and $p\in [1, +\infty)$, we denote by 
$L^p (\sigma)$ the usual Lebesgue space of functions on $\sigma$.
When topologised under the standard norm $\| \cdot \|_{L^p (\sigma)} $ 
it is complete, i.e., a Banach space.
Of course, for $p = 2$ the norm is generated by the standard inner product
$(\cdot,\cdot)_{L^2 (\sigma)}$ and so $L^2 (\sigma)$ is a Hilbert space. 
As usual, the scale $L^p (\sigma)$ continues to include the case $p = \infty$, too.

If $s = 1, 2, \ldots$, we write $H^{s} ({\mathcal X})$  for the Sobolev space of all functions 
$u \in L^2 ({\mathcal X})$ whose generalised partial derivatives up to order $s$ belong 
to $L^2 ({\mathcal X})$. This is a Hilbert space with the standard inner product 
$ (\cdot ,\cdot)_{H^{s} ({\mathcal X})}$. 
The space $H^{s}_{\mathrm{loc}} ({\mathcal X})$ 
consists of functions belonging to $H^{s} (U)$ for 
each relatively compact domain $U \subset {\mathcal X}$. 

Next, for
   $s = 0, 1, \ldots$ and
   $0 \leq \lambda < 1$,
we denote by $C^{s,\lambda} (\overline{\mathcal{X}})$ 
the so-called H\"{o}lder spaces, see for instance
   \cite[Ch.~1, \S~1]{LadSoUr67},
   \cite[Ch.~1, \S~1]{Lady70}.
The normed spaces $C^{s,\lambda} (\overline{\mathcal{X}})$ with
   $s \in \mathbb{Z}_{+}$ and
   $\lambda \in [0,1)$
are known to be Banach spaces which admit the standard embedding theorems.

We are now ready to define proper spaces of periodic functions on ${\mathbb R}^n$.
For this purpose, fix any $\ell > 0$ and denote by ${\mathcal Q}$ be the cube $(0,\ell)^n$
of side length $\ell$. Suppose $s \in \mathbb Z_{+}$. We denote by $H^s$ 
the space of all functions   $u \in H^{s}_\mathrm{loc} ({\mathbb R}^n)$ 
which satisfy the periodicity condition 
\begin{equation}
\label{eq.period}
   u (x + \ell e_j) = u (x)
\end{equation}
for all   $x \in {\mathbb R}^n$ and    $1 \leq j \leq n$, 
where $e_j$ is the $j\,$-th unit basis vector in ${\mathbb R}^n$. 
The space $H^{s}$ is obviously a Hilbert space endowed with the inner product
$
   (u,v)_{H^{s}} = (u,v)_{H^{s} (\mathcal Q)}
$. 
The functions from $H^{s} $ can be easily characterised by their Fourier series
expansions with respect to the orthogonal system 
   $\{ e^{\sqrt{-1} (k,z) (2 \pi / \ell)} \}_{k \in {\mathbb Z}^n}$ 
in $L^2 ({\mathcal Q})$. 
Indeed, as the system consists of eigenfunctions of the Laplace operator $\Delta $ 
corresponding to eigenvalues 
   $\{ \lambda_k = - (k,k) (2 \pi / \ell)^2 \}_{k \in {\mathbb Z}^n}$,
we see that the above scale of Sobolev spaces 
may be defined for all $s \in {\mathbb R}$ by
\begin{equation}
\label{eq.Hs.per}
   H^{s} 
 = \{       u = \sum_{k \in {\mathbb Z}^n} c_k (u) e^{\sqrt{-1} (k,z) (2 \pi / \ell)} :\,
      |c_0 (u)|^2 + \sum_{k \in {\mathbb Z}^n \atop k \ne 0} (k,k)^{s} |c_k (u)|^2 < \infty
   \}
\end{equation}
where $c_k (u)$ are the Fourier coefficients of $u$ with respect to an orthonormal system of 
eigenfunctions of the Laplace operator in the space $\mathbf{L}^2   $ corresponding to the
 eigenvalues $\lambda_k$. 
Traditionally, $\dot{H}^{s} $ stands for the subspace of $H^{s} $ 
consisting of the elements $u$ with $c_0 (u) = 0$ in \eqref{eq.Hs.per}. 
Actually, this discussion leads us to the identification of the space $H^{s} $ 
with Sobolev functions on the torus ${\mathbb T}^n$, to wit, 
   $H^{s}  \cong H^{s} ({\mathbb T}^n)$, see \cite[\S~2.4]{Agra90} and elsewhere.

We also need efficient tools for obtaining a priori estimates.
Namely, it is the Gagliardo-Nirenberg inequality, see \cite{Nir59} for functions 
on ${\mathbb R}^n$. Its analogue for the torus reads as follows
   (see for instance \cite[\S~2.3]{Tema95}).
For $1 \leq  p \leq \infty$, set
\begin{equation*}
   \| \nabla^j u \|_{L^p ({\mathcal Q})}
 := \max_{|\alpha| = j} \| \partial^\alpha u \|_{L^p ({\mathcal Q})}.
\end{equation*}
Then for any function $u \in L^{q_0}  \cap L^{s_0} $ satisfying
   $\nabla^{j_0} u \in L^{p_0} $  and
   $\nabla^{k_0} u \in L^{r_0} $
it follows that
\begin{equation}
\label{eq.L-G-N}
   \| \nabla^{j_0} u \|_{L^{p_0} (\mathcal{Q})}
 \leq
   c_1\, \| \nabla^{k_0} u \|^{a_0}_{L^{r_0} (\mathcal{Q})} \| u \|^{1-a_0 }_{L^{q_0} (\mathcal{Q})}
 + c_2\, \| u \|_{L^{s_0} (\mathcal{Q})}
\end{equation}
whenever
   $s_0 \geq 1$ and
   $0 \leq a_0  \leq 1$,
where
\begin{equation*}
   \frac{1}{p_0}
 = 
    \frac{j_0}{n} + a_0 \left( \frac{1}{r_0} - \frac{k_0}{n} \right) + (1-a_0)\, 
	\frac{1}{q_0},
\,\, 
   \frac{j_0}{k_0}
  \leq 
   a_0,
\end{equation*}
the constants $c_1$ and $c_2$ depend on $j_0$, $k_0$, $s_0$, $p_0$, $q_0$ and $r_0$ 
but not on $u$.

\begin{remark}
\label{r.Delta.r}
Clearly, in the special case $p = 2$ we may always replace the norm
   $\| \nabla^j u \|_{L^p ({\mathcal Q})}$
with the norm
   $\| (- \Delta )^{j/2} u \|_{L^2   }$.
\end{remark}

Next, for $s \in {\mathbb Z}_+$ and $\lambda \in [0,1)$, denote by $C^{s,\lambda} $ 
the space of all functions on ${\mathbb R}^n$ which belong to 
$C^{s,\lambda} (\overline{\mathcal X})$ for any bounded domain 
${\mathcal X} \subset {\mathbb R}^n$ and satisfy \eqref{eq.period}. 
The space $C^\infty $ of spatially periodic $C^\infty$-functions reduces to the intersection 
of the spaces $C^{s,0} $ over $s \in \mathbb{Z}_0$. 
It is endowed with the Fr\'echet topology given by the family of norms 
   $\{ \| u \|_{C^{s,0} } \}_{s \in {\mathbb Z}_+}$. 
Let ${\mathcal D}'$ stand for the space of distributions on ${\mathbb T}^n$, i.e., 
the space of continuous linear functionals on the Fr\'echet space $C^\infty $ endowed 
with the weak topology. 

We  use the  symbol $\mathbf{L}^p$ for the space of periodic vector fields
$u = (u^1,u^2, \dots u^n)$ on $\mathbb{R}^n$ with components $u_i$ in $L^p$.
The space is endowed with the natural norm.
In a similar way we designate the spaces of periodic vector fields on $\mathbb{R}^n$ whose 
components are of Sobolev or H\"{o}lder class. We thus get
   ${\mathbf H}^{s}$ and    ${\mathbf C}^{s,\lambda}$, respectively. 
By  ${\mathbf C}^\infty$ is meant the space of infinitely smooth 
periodic vector fields on $\mathbb{T}^n$. 

To continue, we recall basic formulas of vector analysis saying that
\begin{equation}
\label{eq.deRham}
   \mathrm{rot}\, \nabla
  =
  0, \,
  \mathrm{div}\, \nabla
  =
  \Delta ,
\,
   \mathrm{div}\, \mathrm{rot}
  =
  0,
\,  
  - \mathrm{rot}\, \mathrm{rot} + \nabla\, \mathrm{div}
  =
  E_n \Delta 
\end{equation}
where  $E_n$ is the unit $(n\times n)$-matrix  and $\mathrm{rot}$ is the usual rotation 
operator for $n=3$ and it is the compatibility operator for the gradient operator $\nabla$ in 
the de Rham complex, see, for instance, \cite{deRh55}.

Furthermore, for an integer $s$, we write $V_s$ for the space divergence-free 
vector fields of class ${\mathbf H}^{s}$ and $V'_s$ for the corresponding dual spaces. 
The designations $H$ and $V$ are usually used for $V_0$ and $V_1$, respectively, 
   see \cite[\S~2.1]{Tema95}.  In order to characterize the space $V_s$ we denote by 
$\mathbb{N}_{2,n}$ the set of all natural numbers that can represented as 
$(k,k) = k_1^2 + k_2^2 + \dots + k_n^2$, where $k = (k_1,k_2, \dots k_n)$ 
is a $n$-tuple of natural numbers. For $m \in \mathbb{N}_{2,n}$, let $S_m$ be the 
finite-dimensional linear span of the system  
$\{ e^{\sqrt{-1} (k,z) (2 \pi / \ell)} \}_{(k,k) = m}$ 
of eigenfunctions of the Laplace operator, and let $\mathbf{S}_m$ be the corresponding space
of vector fields on $\mathbb{R}^n$. 

\begin{lemma}
\label{l.basis.V}
Let $m \in {\mathbb N}_{2,n}$.
There are $\mathbf{L}^2$-orthonormal bases
   $\{ v_{m,j} \}_{j=1}^{J_m}$ and
   $\{ w_{m,k} \}_{k=1}^{K_m}$
in the spaces
   $\mathbf{S}_m \cap \ker (\mathrm{div})$ and
   $\mathbf{S}_m \cap \ker (\mathrm{rot})$,
respectively, such that
$$
\begin{array}{rcccl}
   \mathrm{rot} \circ \mathrm{rot}\, v_{m,j}
 & =
 & -\, \Delta  v_{m,j}
 & =
 & m (2 \pi / \ell)^2 v_{m,j},
\\
   \nabla \circ \mathrm{div}\, w_{m,k}
 & =
 & \Delta  w_{m,k}
 & =
 & -\, m (2 \pi / \ell)^2 w_{m,k}
\end{array}
$$
for all
   $j = 1, \ldots, J_m$ and
   $k = 1, \ldots, K_m$. The system $\{ e_1, e_2, \dots e_n, v_{m,j} \}$ 
	is an orthonormal basis in $V_s$.
\end{lemma}

\begin{proof}
The arguments are straightforward. It is worth pointing out that the orthogonality in the  
space $\mathbf{H}^s   $ refers to the inner product 
$$
   (u,v)_s
 = \sum_{i=1}^n \overline{c_{0,i} (v)} c_{0,i} (u) 
 + \sum_{\lambda_j \neq 0} (-\lambda_j)^s\, \overline{c_j (v)} \, c_j (u) ,
$$
where $c_j (u)$ are the Fourier coefficients of $u$ with respect to an orthonormal system of 
eigenfunctions of the Laplace operator in the Lebesgue space 
$\mathbf{L}^2   $ corresponding to the eigenvalues $\lambda_j$ and 
$ c_{0,i} (u) = (u,e_i)_{\mathbf{L}^2} / 
(e_i,e_i)_{\mathbf{L}^2}$. 
\end{proof}

Denote by $\mathbf{P}$ the orthogonal projection of $\mathbf{L}^2 $ onto $V_0$
which is usually referred to as the Helmholtz projection.
By Lemma \ref{l.basis.V}, we get 
\begin{equation*}
   \mathbf{P} = \Pi + \mathrm{rot}^\ast \mathrm{rot}\, \varphi
\end{equation*}
where the operators
   $\Pi$ and
   $\varphi$
are given by
\begin{equation*}
   \Pi u
  =
 c_0 (u), \,
   \varphi u
  =   -\, \sum_{k \in \mathbb{Z}^n \atop k \neq 0}
   \frac{c_k (u)}{(k,k) (2 \pi / \ell)^2}\, e^{\frac{\sqrt{-1} (k,z) 2 \pi}{  \ell}} 
	\mbox{ for } 
	u = \sum_{k \in \mathbb{Z}^n \atop k \neq 0} c_k (u) 
	e^{\frac{\sqrt{-1} (k,z)  2 \pi}{ \ell}}.
\end{equation*}
In particular, 
 $\mathbf{P}$ is actually the orthogonal projection of $\mathbf{H}^s $ onto $V_s$ with 
respect to the unitary structure of $\mathbf{H}^s   $ whenever $s \in {\mathbb Z}_+$.

We will also use the so-called Bochner spaces of functions of $(x,t)$ in the strip
$\mathbb{R}^n \times I$, where $I = [0,T]$.
Namely, if
   $\mathcal B$ is a Banach space (possibly, a space of functions on $\mathbb{R}^n$ and
   $p \geq 1$,
we denote by $L^p (I,{\mathcal B})$ the Banach space of all measurable mappings
   $u : I \to {\mathcal B}$ with finite norm
$$
   \| u \|_{L^p (I,{\mathcal B})}
 := \| \|  u (\cdot,t) \|_{\mathcal B} \|_{L^p (I)},
$$
see for instance \cite[Ch.~III, \S~1]{Tema79}.
In the same line stays the space $C (I,{\mathcal B})$, i.e., it is the Banach space of all
mappings $u : I \to {\mathcal B}$ with finite norm
$$
   \| u \|_{C (I,{\mathcal B})}
 := \sup_{t \in I} \| u (\cdot,t) \|_{\mathcal B}.
$$
We are now in a position to introduce appropriate function spaces for solutions and for the 
data in order to obtain existence theorems for regular solutions to the Navier-Stokes type 
equations \eqref{eq.NS}, cf. similar spaces for non-periodic vector-fields over 
${\mathbb R}^n$ in \cite{ShlTa21} or \cite{Po22} in a more general context of 
elliptic differential complexes.  Namely, for $ k \in {\mathbb Z}_+$, we set 
$$
{\mathcal H}^{(a)}_{k} = \left\{ 
\begin{array}{lll}
\mathbf{H}^{k} & \mbox{ if } & a =0 ,\\
V_{k} & \mbox{ if } & a = 1.
\end{array}
\right.
$$
Then for $s, k \in {\mathbb Z}_+$ we denote by   $B^{k,2s,s}_\mathrm{vel,a} (I)$
the set of all vector fields $u$ in the space 
   $C (I,{\mathcal H}^{(a)}_{k+2s}) \cap L^2 (I,{\mathcal H}^{(a)}_{k+1+2s})$
such that
$$
   \partial_x^\alpha \partial_t^j u
 \in
   C (I,{\mathcal H}^{(a)}_{k+2s-|\alpha|-2j}) \cap L^2 (I,{\mathcal H}^{(a)}_{k+1+2s-|\alpha|-2j})
$$
provided $|\alpha|+2j \leq 2s$.
We endow each space $B^{k,2s,s}_\mathrm{vel,a} (I)$ with the natural norm
$$
   \| u \|_{B^{k,2s,s}_\mathrm{vel,a} (I)}
 :=
   \Big( \sum_{i=0}^k \sum_{|\alpha|+2j \leq 2s}
         \| \partial_x^\alpha \partial_t^j u \|^2_{i,\mu,T}
   \Big)^{1/2}
$$
where 
$
   \displaystyle
   \| u \|_{i,\mu,T}
 = \Big( \| \nabla^i u \|^2_{C (I,\mathbf{L}^2   )}
       + \mu \| \nabla^{i+1} u \|^2_{L^2 (I,\mathbf{L}^2   )}
   \Big)^{1/2}
$
are seminorms on the space $B^{k,2s,s}_\mathrm{vel,a} (I)$, too. 

Similarly, for $s, k \in {\mathbb Z}_+$, we define the space $B^{k,2s,s}_\mathrm{for} (I)$
to consist of all fields $f$ in
$
   C (I,\mathbf{H}^{2s+k}   ) \cap L^2 (I,\mathbf{H}^{2s+k+1}  )
$
with the property that
$
   \partial_x^\alpha \partial _t^j f
 \in
   C (I,\mathbf{H}^{k}   ) \cap L^2 (I,\mathbf{H}^{k+1}   )
$
provided
$
   |\alpha|+2j \leq 2s.
$
If $f \in B^{k,2s,s}_\mathrm{for} (I)$, then actually
$$
   \partial_x^\alpha \partial _t^j f
 \in
   C (I,\mathbf{H}^{k+2(s-j)-|\alpha|}   )
 \cap
   L^2 (I,\mathbf{H}^{k+1+2(s-j)-|\alpha|}   )
$$
for all $\alpha$ and $j$ satisfying $|\alpha|+2j \leq 2s$.
We endow each space $B^{k,2s,s}_\mathrm{for} (I)$ with the natural norm
$$
   \| f \|_{B^{k,2s,s}_\mathrm{for} (I)}
 :=
   \Big( \sum_{i=0}^k \sum_{|\alpha|+2j \leq 2s}
         \| \nabla^i \partial_x^\alpha \partial_t^j f \|^2_{C (I,\mathbf{L}^2({\mathcal Q}))}
       + \| \nabla^{i+1} \partial_x^\alpha \partial_t^j f \|^2_{L^2 (I,\mathbf{L}^2 ({\mathcal Q}))}
   \Big)^{1/2}.
$$
Finally, if  $a =1$, then the space 
$B^{k+1,2s,s}_\mathrm{pre,a}(I)$ for the `pressure'{} $p$ will consist of all functions
   $p \in C (I,\dot{H}^{2s+k+1}   ) \cap L^2 (I,\dot{H}^{2s+k+2}   )$ 
	such that $\nabla p \in B^{k,2s,s}_\mathrm{for} (I)$. 
Obviously, the space does not contain functions depending on $t$ only, 
and this allows us to equip it with the norm
$$
   \| p \|_{B^{k+1,2s,s}_\mathrm{pre,a} (I)}
 = \| \nabla p \|_{B^{k,2s,s}_\mathrm{for} (I)}.
$$
If $a=0$ then we define the space for the `pressure'{} $p$ as 
$B^{k+1,2s,s}_\mathrm{pre,0} (I) = \{ 0\}$. 

It is easy to see that
   $B^{k,2s,s}_\mathrm{vel,a} (I)$,
   $B^{k,2s,s}_\mathrm{for} (I)$ and
   $B^{k+1,2s,s}_\mathrm{pre,a} (I)$
are Banach spaces with the following obvious 
properties.

\begin{lemma}
\label{l.NS.cont.0}
Suppose that
   $s \in \mathbb N$, $k \in {\mathbb Z}_+$. 
The following mappings are continuous:
$$
   \nabla :
  B^{k+1,2(s-1),s-1}_\mathrm{pre,a} (I)
  \to
  B^{k,2(s-1),s-1}_\mathrm{for} (I),
$$
$$
   \partial_j \partial_i :
  B^{k,2s,s}_\mathrm{vel,a} (I)
  \to
  B^{k,2(s-1),s-1}_\mathrm{for} (I), \,
  \partial_j \partial_i  :
  B^{k+2,2(s-1),s-1}_\mathrm{vel,a} (I)
  \to
  B^{k,2(s-1),s-1}_\mathrm{for} (I),
$$
$$
   \partial_t :
  B^{k,2s,s}_\mathrm{vel,a} (I)
  \to
  B^{k,2(s-1),s-1}_\mathrm{for} (I),
\,
 \delta_t :
  B^{k,2s,s}_\mathrm{vel,a} (I)
  \to
  {\mathcal H}^{(a)}_{k+2s},
$$
where $\delta_t (u (\cdot,t)) = u (\cdot,0)$ is the initial value functional 
(or the delta-function in $t$). 
\end{lemma}
 
\begin{proof} Follows immediately from the definition of the spaces. 
\end{proof} 

\begin{lemma} \label{l.emb.Bochner}
If $s\in \mathbb N$ then the  following embedding are continuous:
$$
B^{k,2s,s}_{\mathrm{vel,a}} (I)  \hookrightarrow  
B^{k+2,2(s-1),s-1}_{\mathrm{vel,a}}(I) ,  \,\,
B^{k,2s,s}_{\mathrm{for}} (I)\hookrightarrow  
B^{k+2,2(s-1),s-1}_{\mathrm{for}} (I) , 
$$
$$
B^{k+1,2s,s}_{\mathrm{pre,a}}(I)  \hookrightarrow  
B^{k+3,2(s-1),s-1}_{\mathrm{pre,a}} (I) .
$$
If, in addition, $k+2s>n/2-1$ then the embeddings
$$
B^{k,2s,s}_{\mathrm{vel,a}} (I)\hookrightarrow  
L^\infty (I, \mathbf{L}^n )  , \,\, 
B^{k,2s,s}_{\mathrm{vel,a}} \hookrightarrow  
L^{\mathfrak{s}} (I, \mathbf{L}^{\mathfrak{r}})
$$
are continuous, too, if $\mathfrak{s}$, $\mathfrak{r}$ satisfy \eqref{eq.s.r}.
\end{lemma}

\begin{proof} The continuity of the first three embeddings 
follows immediately from the definition of the spaces. 
The other embeddings follow from \eqref{eq.L-G-N} and 
the Sobolev embedding theorem (see, for instance, \cite[Ch.~4, Theorem~4.12]{Ad03}).
For the fifth embedding with $\mathfrak{s}=2$, $\mathfrak{r}=+\infty$,
 we use the following: if   $k, s \in {\mathbb Z}_+$ and 
   $\lambda\in (0,1)$ satisfying    $k - s - \lambda > n/2$, 
the there exists a constant $c (k,s,\lambda)$ depending on the parameters, such that 
for all $u \in H^{k}   $ we have 
\begin{equation}
\label{eq.Sob.index}
   \| u \|_{C^{s,\lambda}    }
 \leq c (k,s,\lambda)\, \| u \|_{H^{k}   }
\end{equation}
 cf. \cite[Lemma 3.4]{ShlTa21}
for vector-fields over ${\mathbb R}^n$ without periodicity assumptions. 
\end{proof}

We now proceed with studying \eqref{eq.NS} in these spaces.

\section{An open mapping theorem}
\label{s.OM}

This section is devoted to the so-called stability property for solutions to the Navier-Stokes 
type equations \eqref{eq.NS}. One of the first statements of this kind was obtained  by
   O.A.Ladyzhenskaya \cite[Ch.~4, \S~4, Theorem~11]{Lady70} 
for flows in bounded domains in ${\mathbb R}^3$ with $C^2$ smooth boundaries.

In order to extend the property to the spaces of high smoothness, we consider the standard 
linearisation of problem \eqref{eq.NS} at the zero solution $(0,0)$. Namely, let 
$$
   \mathbf{B} (w,u)
 = {\mathcal M} (u,w) + 	{\mathcal M} (w,u)
$$
for vector fields    $u = (u^1, u^2, \dots, u^n)$ and
   $w = (w^1, w^2, \dots w^n)$.  

We proceed with a simple lemma.

\begin{lemma}
\label{l.NS.cont}
Suppose that $s \in \mathbb N$, 
   $k \in {\mathbb Z}_+$, $2s+k>\frac{n}{2}-1$,  and
   $w \in B^{k,2s,s}_\mathrm{vel,a} (I)$. 
The following mappings are continuous:
$$
\begin{array}{rrcl}
     \mathbf{B} (w,\cdot) :
 & B^{k+2,2(s-1),s-1}_\mathrm{vel,a} (I)
 & \to
 & B^{k,2(s-1),s-1}_\mathrm{for} (I),
\\
  \mathbf{B} (w,\cdot) :
 & B^{k,2s,s}_\mathrm{vel,a} (I)
 & \to
 & B^{k,2(s-1),s-1}_\mathrm{for} (I),
\\[.2cm]
   \mathbf{D} :
 & B^{k+2,2(s-1),s-1}_\mathrm{vel,a} (I)
 & \to
 & B^{k,2(s-1),s-1}_\mathrm{for} (I),
\\
   \mathbf{D} :
 & B^{k,2s,s}_\mathrm{vel,a} (I)
 & \to
 & B^{k,2(s-1),s-1}_\mathrm{for} (I).
\end{array}
$$
Moreover, with positive constants $c_{s,k} $ independent on $u,w$, we have 
\begin{equation}\label{eq.B.pos.bound}
\|  \mathbf{B} (w,u)\|_{B^{k,2(s-1),s-1}_\mathrm{for} (I)}
\leq c_{s,k} 
	\|w\|_{B^{k+2,2(s-1),s-1}_\mathrm{vel,a} (I)} 
	\|u\|_{B^{k+2,2(s-1),s-1}_\mathrm{vel,a} (I)}.
\end{equation}
\end{lemma}

\begin{proof} Follows easily from the very definition of the spaces and Gagliardo-Nirenberg 
inequality \eqref{eq.L-G-N}, cf. \cite[Lemma 3.5]{ShlTa21} for vector-fields 
over ${\mathbb R}^n$ without periodicity assumptions; cf. also \cite[Theorem 1.4]{Po22} 
for spaces on compact closed manifolds.
\end{proof}

Now, let us consider a linearisation of problem \eqref{eq.NS}:
given spatially periodic functions
   $f = (f^1, f^2, \dots f^n) \in B^{k,2(s-1),s-1}_\mathrm{for,a} (I)$,
   $w = (w^1, w^2 , \dots w^n) \in B^{k,2s,s}_\mathrm{vel,a} (I)$
on ${\mathbb R}^n \times [0,T]$ and
   $u_0 = (u^1_{0}, u^2_{0}, \dots u^n_{0}) \in {\mathcal H}^{(a)}_{2s+k} $ 
on ${\mathbb R}^n$ with values in $\mathbb{R}^n$, find spatially 
periodic functions
   $u = (u^1, u^2, \dots u^n) \in B^{k,2s,s}_\mathrm{vel,a} (I)$ and
   $p \in B^{k+1,2s,s}_\mathrm{pre,a} (I)$
in the strip ${\mathbb R}^n \times [0,T]$ which satisfy
\begin{equation}
\label{eq.NS.lin}
\left\{
\begin{array}{rcll}
   \partial _t u   -  \mu \Delta  u 	+ 	{\mathbf B} (w,u)	+ 	a \, \nabla p
 & =
 & f,
 & (x,t) \in {\mathbb R}^n \times (0,T),
\\
  a \, \mbox{div}\, u
 & =
 & 0,
 & (x,t) \in {\mathbb R}^n \times (0,T),
\\[.05cm]
   u
& =
& u_0,
& (x,t) \in \mathbb{R}^n \times \{ 0 \}.
\end{array}
\right.
\end{equation}

Considering this problem in the Bochner spaces yields the expected existence and 
uniqueness theorem. 

\begin{theorem}[\bf A. Shlapunov, N. Tarkhanov]
\label{t.exist.NS.lin.strong}
Let
   $s \in \mathbb N$,
   $k \in {\mathbb Z}_+$, $2s+k>\frac{n}{2}$, 
and
   $w \in B^{k,2s,s}_\mathrm{vel,a} (I)$.
Then \eqref{eq.NS.lin} induces a bijective continuous linear mapping
\begin{equation}
\label{eq.map.Aw}
   \mathcal{A}_{w,a} :
   B^{k,2s,s}_\mathrm{vel,a} (I) \times B^{k+1,2(s-1),s-1}_\mathrm{pre,a} (I)
\to
   B^{k,2(s-1),s-1}_\mathrm{for} (I) \times {\mathcal H}^{(a)}_{2s+k},
\end{equation}
which admits a continuous inverse $\mathcal{A}^{-1}_{w,a}$.
\end{theorem}

\begin{proof}  The continuity of $\mathcal{A}_{w,a}$ follows from Lemmata 
\ref{l.NS.cont.0}, \ref{l.NS.cont}. 
Next, one usually follows a rather standard scheme beginning with the notion of a 
weak solution, 
see, for instance \cite[Ch. VI, \S 5]{Lady70}, \cite[Ch. 3, \S 1]{Tema79} 
or \cite[Theorems 3.1 and 3.2]{ShlTa21} for this particular 
type of spaces in the case of vector fields over ${\mathbb R}^n \times (0,T)$ 
without periodicity assumptions. 
Namely, one usually begins with the following statement. 

\begin{proposition}
\label{p.exist.NS.lin.weak}
Suppose $w \in C (I,{\mathcal H}^{(a)}_0) \cap L^2 (I,{\mathcal H}^{(a)}_{1}) 
\cap L^{2} (I,\mathbf{L}^{\infty}   )\cap L^{\infty} (I,\mathbf{L}^{n}   )$.
Given any pair $(f,u_0) \in L^2 (I,({\mathcal H}^{(a)}_{1})') \times {\mathcal H}^{(a)}_{0}$, there is a unique vector field
   $u \in C (I,{\mathcal H}^{(a)}_{0}) \cap L^2 (I,{\mathcal H}^{(a)}_{1})$ with
   $\partial_t u \in L^2 (I,({\mathcal H}^{(a)}_{1})')$,
satisfying for all $v \in  {\mathcal H}^{(a)}_{1}$
\begin{equation}
\label{eq.NS.lin.weak}
\left\{
   \begin{array}{rcl}
   \displaystyle
   \frac{d}{dt} (u,v)_{\mathbf{L}^2   }
 + \mu  (\nabla u, \nabla v)_{\mathbf{L}^2   }  & =
 & \langle f - \mathbf{B} (w,u), v \rangle,
\\
   u (\cdot,0)
 & =
 & u_0.
   \end{array}
\right.
\end{equation}
\end{proposition}

\begin{proof}
It is similar to the proof of the uniqueness and existence theorem for the Stokes problem and
the Navier-Stokes problem, see
   \cite[\S~2.3, \S~2.4]{Tema95}
(or
   \cite[Ch.~II, Theorem 6.1 and Theorem 6.9]{Lion69} or
   \cite[Ch.~III, Theorem 1.1, Theorem 3.1 and Theorem 3.4]{Tema79}
for domains in ${\mathbb R}^3$). 
It is based on Gronwall type Lemma \ref{l.Perov} with $0<\gamma_0\leq 1$, 
inequality \eqref{eq.L-G-N} and the following useful lemmata and formulas.

\begin{lemma}
\label{l.Lions}
Let $V$, $H$ and $V'$ be Hilbert spaces such that $V'$ is the dual to $V$ and the embeddings
$
   V \subset H \subset V'
$
are continuous and everywhere dense.
If
   $u \in L^2 (I,V)$ and
   $\partial_t u \in L^2 (I,V')$
then
\begin{equation*}
   \frac{d}{dt} \| u (\cdot, t) \|^2_{H}
 = 2\, \langle \partial_t u, u \rangle
\end{equation*}
and $u$ is equal almost everywhere to a continuous mapping from $[0,T]$ to $H$.
\end{lemma}

\begin{proof}
See \cite[Ch.~III, \S~1, Lemma~1.2]{Tema79}.
\end{proof}

The following standard statement, where
\begin{eqnarray*}
   \| u \|_{k,\mu,T}
 & =
 & \Big( \| \nabla^k u \|^2_{C (I, \mathbf{L}^2   )}
       + \mu\, \| \nabla^{k+1} u \|^2_{L^2 (I, \mathbf{L}^2   )}   \Big)^{1/2},
\\
   \| (f,u_0) \|_{0,\mu,T}
 & =
 & \Big( \| u_0 \|^2_{\mathbf{L}^2   }
       + \frac{2}{\mu}\, \| f \|^2_{L^2 (I, ({\mathcal H}^{(a)}_1)')}
       + \| f \|^2_{L^1 (I, ({\mathcal H}^{(a)}_1)')}
   \Big)^{1/2},
\end{eqnarray*}
gives a basic a priori estimate for weak solutions to 
\eqref{eq.NS.lin.weak}. 

\begin{lemma}
\label{p.En.Est.u}
Let
   $w \in L^2 (I,{\mathcal H}^{(a)}_1) \cap C (I,{\mathcal H}^{(a)}_0) \cap 
	L^{2} (I, \mathbf{L}^{\infty}   )$.
If
   $u \in C (I,{\mathcal H}^{(a)}_0) \cap L^2 (I,{\mathcal H}^{(a)}_1)$ and
   $(f,u_0) \in L^2 (I,({\mathcal H}^{(a)}_1)') \times {\mathcal H}^{(a)}_0$
satisfy
\begin{equation}
\label{eq.En.equal}
\left\{
\begin{array}{rcl}
   \displaystyle
   \frac{1}{2} \frac{d}{d\tau}\, \| u (\cdot,\tau) \|^2_{\mathbf{L}^2   }
 + \mu\, \| \nabla u \|^2_{\mathbf{L}^2   } 
 & \leq 
 & \langle f - \mathbf{B} (w,u), u \rangle 
\\
   u (\cdot,0)
 & =
 & u_0
\end{array}
\right.
\end{equation}
for all $t \in [0,T]$, then
\begin{equation}
\label{eq.En.Est2}
\begin{array}{rcl}
   \displaystyle
   \| u \|^2_{0,\mu,T}
 & \leq &
   \displaystyle
   \| (f,u_0) \|^2_{0,\mu,T}
   \Big(
     1
   + c_1 
     \exp \Big( \frac{c_2}{\mu} \int_0^T \| w (\cdot,t) \|^{2}_{\mathbf{L}^{\infty}   } dt \Big)
\\
 & + &
   \displaystyle
   \frac{c_3}{\mu}
   \Big( \int_0^T \| w (\cdot,t) \|^{2}_{\mathbf{L}^{\infty}   } dt \Big)
   \exp \Big( \frac{c_4}{\mu} \int_0^T \| w (\cdot,t)\|^{2}_{\mathbf{L}^{\infty}   } dt \Big)
\Big)   
\end{array}
\end{equation}
with positive constant $c_j$ independent on $w$ and $u$. 
\end{lemma}

It is easy to see that  for any $p \geq 1$ we have
\begin{equation}
\label{eq.En.Est2add}
   \| u \|_{L^{p} (I, \mathbf{L}^2   )}
 \leq T^{1/p}\, \| u \|_{L^\infty (I,\mathbf{L}^2   )}
\end{equation}
which accomplishes the energy estimate \eqref{eq.En.Est2}. 
The rest of the proof runs the standard scheme with the use of Faedo-Galerkin 
approximations, see \textit{ibid}. 
\end{proof}

Finally, one has to invoke the standard arguments related to the increasing of the 
regularity of weak solution $u$ for perturbations of Stokes equations, see, for instance 
\cite[Ch. VI, \S 5]{Lady70}, \cite[Ch. 3, \S 1]{Tema79} 
or \cite[Theorems 3.1 and 3.2]{ShlTa21} for this particular 
type of spaces in the case of vector fields over ${\mathbb R}^n \times (0,T)$ 
without periodicity assumptions. Of course, if $a=1$ then we additionally have to recover 
the `pressure'{} $p$ via known `velocity'{} $u$. Then the following statement can be used.

\begin{proposition}
\label{c.Sob.d}
Let $n\geq 3$, $s \in {\mathbb N}$, $k \in {\mathbb Z}_+$,  
and $F\in B^{k,2(s-1),s-1}_{\mathrm{for}}$ 
satisfies $\mathbf{P}F    =0$. 
Then there is a unique function $p  \in  B^{k+1,2(s-1),s-1}
_{\mathrm{pre,1}} $, satisfying 
\begin{equation*} 
\nabla p =F \mbox{ in } {\mathbb T}^n \times [0,T].
\end{equation*}
\end{proposition}

\begin{proof}
According to \cite{deRh55}, for any $w \in \mathbf{H}^{s} $, the following statements are equivalent: 
1) $w = \nabla p$ for some function $p \in H^{s+1} $; 
2) $w \in \dot{H}^{s}  \cap \ker (\mathrm{rot})$; 
3) $\mathbf{P} w = 0$. 

Hence it follows, in particular, that $\nabla$ establishes an isomorphism between
   $\dot{H}^{s+1} $ and
   $\dot{\mathbf{H}}^{s}  \cap \ker (\mathrm{rot})$.
Then the statement of the proposition follows actually from the Hodge formula
$$
   - \mathrm{rot}\,  \mathrm{rot}\,  \varphi  \, w 
   + \nabla \, \mathrm{div}\, \varphi\,  w 
 = w - \Pi w,
$$
i.e. given $F\in B^{k,2(s-1),s-1}_{\mathrm{for}}$ we have  $p =  \mathrm{div} \, \varphi\, F$ 
for $F\in B^{k,2(s-1),s-1}_{\mathrm{for}}$ satisfying $\mathbf{P}F    =0$. 
Then the continuity of the operator $ \mathrm{div}\, \varphi:
{\mathbf H}^{s-1} \to \dot{H}^{s}  $ on the scale of the Sobolev
spaces means that for $0\leq j \leq s-1$ we have 
$$
\sup_{t \in T}\|\partial _t ^j p\|_{H^{2s-1+k-2j}} \leq 
c\, \sup_{t \in T}\|\partial _t ^j F\|_{H^{2(s-1)+k-2j}},
$$
$$
\int_0^T \|\partial _t ^j p\|^2_{H^{2s-1+k-2j}} \, d\tau \leq 
c\, \int_0^T\|\partial _t ^j F\|2_{H^{2(s-1)+k-2j}} \,
$$
with a positive constant $c$ independent on $p$ and $F$, i.e. $p  \in  B^{k+1,2(s-1),s-1}
_{\mathrm{pre,1}} $. 
\end{proof}

This finishes our sketch of the proof of Theorem \ref{t.exist.NS.lin.strong}.  
\end{proof}

Since problem \eqref{eq.NS.lin} is a linearisation of the Navier-Stokes type equations at an 
arbitrary vector field $w$, it follows from Theorem \ref{t.exist.NS.lin.strong} that the 
nonlinear mapping given by the Navier-Stokes type equations is locally invertible.
The implicit function theorem for Banach spaces even implies that the local inverse mappings 
can be obtained from the contraction principle of Banach.
In this way we obtain what we shall call the open mapping theorem for problem \eqref{eq.NS}.

\begin{theorem}[\bf A. Shlapunov, N. Tarkhanov]
\label{t.open.NS.short}
Let
   $s \in \mathbb N$ and
   $k \in {\mathbb Z}_+$ satisfy $2s+k>\frac{n}{2}$. 
Then \eqref{eq.NS} induces an injective continuous open nonlinear mapping
\begin{equation}
\label{eq.map.A}
   \mathcal{A}_{a} :
   B^{k,2s,s}_\mathrm{vel,a} (I) \times B^{k+1,2(s-1),s-1}_\mathrm{pre,a} (I)
\to
   B^{k,2(s-1),s-1}_\mathrm{for} (I) \times \mathcal{H}^{(a)}_{2s+k}. 
\end{equation}
\end{theorem}

The principal significance of the theorem is in the assertion that for each point
   $(u_0,p_0) \in B^{k,2s,s}_\mathrm{vel,a} (I) \times B^{k+1,2(s-1),s-1}_\mathrm{pre,a} (I)$
there is a neighbourhood $\mathcal{V}$ of the image $\mathcal{A}_a (u_0,p_0)$ in
   $B^{k,2(s-1),s-1}_\mathrm{for} (I) \times {\mathcal H}^{(a)}_{2s+k}$,
such that $\mathcal{A}_a$ is a homeomorphism of the open set
   $\mathcal{U} := \mathcal{A}_a^{-1} (\mathcal{V})$
onto $\mathcal{V}$.

\begin{proof}
Indeed, the continuity of the mapping $\mathcal{A}_a$ is clear from Lemma \ref{l.NS.cont}.
Let 
   $(u',p')$ and
   $(u'',p'')$
belong to
   $B^{k,2s,s}_\mathrm{vel,a} (I) \times B^{k+1,2(s-1),s-1}_\mathrm{pre,a} (I)$
and
   $\mathcal{A}_a (u',p') = \mathcal{A}_a (u'',p'')$. 
Integrating by parts, we easily see that 
\begin{equation}\label{eq.nabla.p}
(a \, \nabla p, v)_{\mathbf{L}^2} = 0
\mbox{ for all } v \in {\mathcal H}^{(a)}_{1}
\end{equation}
if $a=0$ or $a=1$ and $p \in B^{k+1,2(s-1),s-1}_\mathrm{pre,a} (I)$. 
Then,  Lemma \ref{l.Lions}, formula \eqref{eq.nabla.p}
and an integration by parts with the use of \eqref{eq.deRham} 
imply that for $v= u'-u''$ we have 
\begin{equation}
\label{eq.NS.weak.u}
   \frac{1}{2}\frac{d}{dt} \|v\|^2_{\mathbf{L}^2   }
 + \mu \|\nabla v\|^2_{\mathbf{L}^2   } 
  =
  ( \mathbf{D} u' - \mathbf{D} u'', v )_{\mathbf{L}^2}.
\end{equation}
Since    the bilinear form $\mathbf{B}$ is symmetric and
   $\mathbf{B} (u,u) = 2 \mathbf{D} (u)$,
we easily obtain
\begin{equation}
\label{eq.M.diff}
   \mathbf{D} (u') - \mathbf{D} (u'')
 = \mathbf{B} (u', u'-u'') + (1/2)\, \mathbf{B} (u'-u'', u'-u'').
\end{equation}
By the Sobolev embedding theorem, see \eqref{eq.Sob.index}, the space
   $B^{k,2s,s}_\mathrm{vel,a} (I) $
is continuously embedded into $L^2 (I,\mathbf{L}^\infty )$ 
if $2s+k >\frac{n}{2}$ and then, after an integration by parts, 
\begin{equation} \label{eq.uniq.est}
\int_0^t \left| ( \mathbf{D} u' - \mathbf{D} u'', v )_{\mathbf{L}^2} \right| 
\, d\tau \leq 
C\, \int_0^t \|\nabla v\|_{\mathbf{L}^2} \| v\|_{\mathbf{L}^2}
\Big(\|u'\|_{\mathbf{L}^\infty}  + \|v\|_{\mathbf{L}^\infty} \Big) \, d\tau \leq 
\end{equation}
$$
\int_0^t \Big(\frac{\mu}{2}\, \|\nabla v\|^2_{\mathbf{L}^2} 
+ C_\mu \Big(\|u'\|^2_{\mathbf{L}^\infty}  + \|v\|^2_{\mathbf{L}^\infty} \Big) \,  \| v\|^2_{\mathbf{L}^2} \Big) \, d\tau 
$$
for all $t \in [0,T]$ with positive constants $C$ and $C_\mu $ independent 
on $u'$ and $u''$. 
Hence, as $v(x,0)=0$, combining \eqref{eq.NS.weak.u}, \eqref{eq.uniq.est} with Gronwall 
type Lemma \ref{l.Perov} for $0<\gamma_0\leq 1$, we conclude that $v=0$, i.e.  $u' = u''$. 

If $a=1$ then 
   $\nabla (p' - p'') (\cdot,t) = 0$
for all $t \in [0,T]$.
It follows that the difference $p'-p''$ is identically equal to a function $c (t)$ on
the segment $[0,T]$.
Since $p'-p'' \in C (I,\dot{H}^{k})$, we conclude by Proposition \ref{c.Sob.d} that
   $p'-p'' \equiv 0$.
So, the operator $\mathcal{A}_a$ of \eqref{eq.map.A} is injective for both $a=0$ and $a=1$.

Now, equality \eqref{eq.M.diff}  makes it evident that the Frech\'et derivative
$\mathcal{A}'_{a,(w,p_0)}$ of the nonlinear mapping $\mathcal{A}_a$ at an arbitrary point 
$
   (w,p_0)
 \in B^{k,2s,s}_\mathrm{vel,a} (I) \times B^{k+1,2(s-1),s-1}_\mathrm{pre,a} (I)
$ 
coincides with the continuous linear mapping $\mathcal{A}_{a,w}$ of \eqref{eq.map.Aw}.
By Theorem \ref{t.exist.NS.lin.strong}, $\mathcal{A}_{a,w}$ is an invertible continuous linear
mapping from
   $B^{k,2s,s}_\mathrm{vel,a} (I) \times B^{k+1,2(s-1),s-1}_\mathrm{pre,a} (I)$ to
   $B^{k,2(s-1),s-1}_\mathrm{for} (I) \times {\mathcal H}^{(a)}_{k+2s}$.
Both the openness of the mapping $\mathcal{A}_a$ and the continuity of its local 
inverse mapping  follow now from the implicit function theorem for Banach spaces,
   see for instance  \cite[Theorem 5.2.3, p.~101]{Ham82}. 
	Actually, we need the following simple 
corollary of this theorem: Let $A: B_1 \to B_2$ be an  everywhere defined 
	smooth (admitting the Fr\'echet derivative at each point) map 
	between Banach spaces $B_1$, $B_2$. If for some point $v_0$ the Fr\'echet 
	derivative $A'_{|v_0}$ is a linear continuously invertible map then we can find 
	a neighbourhood $V$ of $v_0$ and 	a neighbourhood $U$ of $g_0 = A(v_0)$ 
	such that the map $A_a$ gives a one-to-one map of $V$ onto $U$ and the (local) 
	inverse map $A^{-1}: U \to V$ is continuous and smooth. 
\end{proof}

Thus, we identify the Navier-Stokes type equations \eqref{eq.NS} 
with the following operator equation related to the mapping \eqref{eq.map.A}:
given data $f \in B^{k,2(s-1),s-1}_\mathrm{for} (I)$ and $u_0 \in {\mathcal H}^{(a)}
_{k+2s}$ find a pair $(u,p) \in B^{k,2s,s}_\mathrm{vel,a} (I) \times B^{k+1,2(s-1),s-1}_
\mathrm{pre,a} (I)$ satisfying 
\begin{equation} \label{eq.NS.map}
{\mathcal A}_a (u,p) = (f,u_0).
\end{equation}

Then theorem \ref{t.open.NS.short} suggests a clear direction for the development of the 
topic. 

\begin{corollary}
\label{c.clopen}
Let
   $s \in \mathbb N$ and
   $k \in {\mathbb Z}_+$, $2s+k>\frac{n}{2}$. 
The range of the mapping \eqref{eq.map.A} is closed if and only if it coincides with the whole
destination space.
\end{corollary}

\begin{proof}
Since the destination space is convex, it is connected. As is known, the only clopen 
(closed and open) sets in a connected topological vector space are the empty set and the 
space itself. Hence, the range of the mapping $\mathcal{A}_a$ is closed if and only if it 
coincides with the whole destination space.
\end{proof}

\section{A surjectivity criterion}
\label{s.surjective}

Inspired by \cite{Pro59}, \cite{Serr62}, \cite{Lady70} and \cite{Lion61,Lion69},
let us obtain a surjectivity criterion  for mapping \eqref{eq.map.A} in terms of 
$L^{\mathfrak{s}} (I,\mathbf{L}^{\mathfrak{r}}   )\,$-estimates 
for solutions to \eqref{eq.NS.map} via the data. 

\begin{theorem}[\bf A. Shlapunov, N. Tarkhanov]
\label{t.sr}
Let  $s \in \mathbb{N}$,
   $k \in {\mathbb Z}_+$, $2s+k>\frac{n}{2}$, 
	the numbers $\mathfrak{r}$, $\mathfrak{s}$ satisfy \eqref{eq.s.r} and 
\begin{equation} \label{eq.trilinear}
	({\mathbf D} v, v)_{\mathbf{L}^2} \geq 0
\mbox{ for all } v \in {\mathcal H}^{(a)}_{2s+k} .
\end{equation}
Then mapping \eqref{eq.map.A} is surjective 
if and only if, given subset   $S = S_\mathrm{vel,a} \times S_\mathrm{pre,a}$ of the product
   $B^{k,2s,s}_\mathrm{vel,a} (I) \times B^{k+1,2(s-1),s-1}_\mathrm{pre,a} (I)$ 
such that the image $\mathcal{A} (S)$ is precompact in the space
   $B^{k,2(s-1),s-1}_\mathrm{for} (I) \times {\mathcal H}^{(a)}_{2s+k}$,
the set $S_\mathrm{vel,a}$ is bounded in the space
   $L^{\mathfrak{s}} (I,\mathbf{L}^{\mathfrak{r}}   )$.
\end{theorem}

\begin{proof} Let mapping \eqref{eq.map.A} be surjective. Then the range of this mapping 
is closed according to Theorem \ref{t.open.NS.short}. Fix a subset   $S = S_\mathrm{vel,a} \times S_\mathrm{pre,a}$ of the product
   $B^{k,2s,s}_\mathrm{vel,a} (I) \times B^{k+1,2(s-1),s-1}_\mathrm{pre,a} (I)$ 
such that the image $\mathcal{A}_a (S)$ is precompact in 
   $B^{k,2(s-1),s-1}_\mathrm{for} (I) \times {\mathcal H}^{(a)}_{2s+k}$. 
If the set $S_\mathrm{vel}$ is unbounded in the space
   $L^{\mathfrak{s}} (I,\mathbf{L}^{\mathfrak{r}}   )$ then 
	there is a sequence $\{ (u_i,p_i) \} \subset S$ such that
\begin{equation}
\label{eq.unbounded}
   \lim_{i \to \infty} \| u_i \|_{L^{\mathfrak{s}} 
(I,\mathbf{L}^{\mathfrak{r}}   )}
 = \infty.
\end{equation}
As the set $\mathcal{A}_a (S)$ is precompact in
   $B^{k,2(s-1),s-1}_\mathrm{for} (I) \times {\mathcal H}^{(a)}_{2s+k}$,
we conclude that the corresponding sequence of data
   $\{ \mathcal{A}_a (u_i,p_i) = (f_i,u_{i,0})\}$
contains a subsequence $\{ (f_{i_m},u_{i_m,0})\}$ which converges to an element
   $(f,u_0) $ in this space. But the range of the map is closed and hence 
for the data $(f,u_0)$ there is a unique solution $(u,p)$ to 
\eqref{eq.NS.map} 
in the space
   $B^{k,2s,s}_\mathrm{vel,a} (I) \times B^{k+1,2(s-1),s-1}_\mathrm{pre,a} (I)$
and the sequence $\{ (u_{i_m},p_{i_m}) \}$ converges to $(u,p) $ in this space.
Therefore, $\{ (u_{i_m},p_{i_m}) \}$ is bounded in
   $B^{k,2s,s}_\mathrm{vel,a} (I) \times B^{k+1,2(s-1),s-1}_\mathrm{pre,a} (I)$
and this contradicts \eqref{eq.unbounded} because the space $B^{k,2s,s}_\mathrm{vel,a} (I)$
 is embed\-ded continuously into the space $L^{\mathfrak{s}} 
(I,\mathbf{L}^{\mathfrak{r}}   )$ for any pair $\mathfrak{r}$, $\mathfrak{s}$ satisfying 
\eqref{eq.s.r}.

We continue with typical estimates for 
solutions to operator equati\-on \eqref{eq.NS.map}.

\begin{lemma}
\label{p.En.Est.u.strong}
Let \eqref{eq.trilinear} be fulfilled. If
   $(u,p) \in B^{0,2,1}_{\mathrm{vel,a}} (I) \times B^{1,0,0}_{\mathrm{pre,a}} (I)$
is a solution to the Navier-Stokes type equations \eqref{eq.NS.map} with data
   $(f,u_0) \in B^{0,0,0}_{\mathrm{for}} (I) \times {\mathcal H}^{(a)}_2$,
then
\begin{equation}
\label{eq.En.Est2imp}
   \| u \|_{0,\mu,T}
 \leq
c_0   \| (f,u_0) \|_{0,\mu,T}
\end{equation}
with a positive constant $c_0$ independent on $u$, $f$ and $u_0$.
\end{lemma}

\begin{proof}
For a solution  $(u,p)$ to \eqref{eq.NS.map} related to 
data $(f,u_0)$ within the declared function classes, the component   
$u$ belongs to $C (I, \mathbf{H}^{2}) \cap L^2 (I, \mathbf{H}^{3})$, and both $\partial_t u $ 
and $f$  belong to $C (I, \mathbf{L}^{2}) \cap L^2 (I, \mathbf{H}^{1})$. 
Next, we may calculate the inner product $({\mathcal A}_a u, u)_{\mathbf{L}^2}$
with the use \eqref{eq.trilinear}, \eqref{eq.nabla.p} 
and Lemma \ref{l.Lions}, obtaining 
\begin{equation*}
 \frac{1}{2}\frac{d}{dt} \|u\|^2_{\mathbf{L}^2}
 + \mu \| \nabla u\|^2_{\mathbf{L}^2} 	\leq \langle f , u \rangle   
	\end{equation*}
Finally, applying Lemma \ref{p.En.Est.u}  with $w = 0$, we conclude that 
the estimate follow.
\end{proof}

We note that estimate 
\eqref{eq.En.Est2imp} corresponds to the basic Energy Estimate for weak solutions to 
\eqref{eq.NS}, see, for instance, \cite[\S~2.3, \S~2.4]{Tema95};
it is slightly stronger than the standard  one because 
the elements of the space $B^{0,2,1}_{\mathrm{vel}} (I)$ are already rather regular.
Thus, the estimate implies neither an Existence Theorem for operator equation 
\eqref{eq.NS.map} nor an improvement of the regularity of weak solutions to 
\eqref{eq.NS}.

Next, we obtain some estimates for the derivatives vector fields with respect to the space 
variables.

\begin{lemma}
\label{l.En.Est.Du.L2}
Let $n\geq 2$, 
   $k \in \mathbb{Z}_+$,  $k+2>\frac{n}{2}$, 
and $\mathfrak{s}$,   $\mathfrak{r}$ satisfy \eqref{eq.s.r}.
Then for any
   $\varepsilon > 0$
and for all
   $u \in \mathbf{H}^{2+k} $
it follows that for all $0\leq k'\leq k$ we have
\begin{equation}
\label{eq.En.Est3bbb}
   \| (- \Delta)^{\frac{k'}{2}} \mathbf{D} u \|^2_{\mathbf{L}^2   }
 \, \leq \,
   \varepsilon\, \| \nabla^{k'+2} u \|^2_{\mathbf{L}^2    }
+ c (k',\mathfrak{s},\mathfrak{r},\varepsilon)\, \| u \|^{\mathfrak{s}}
_{\mathbf{L}^\mathfrak{r} } \|\nabla^{k'+1} u \|^2_{\mathbf{L}^2   } +
 \end{equation}                                               
\begin{equation*}
   c (k',\mathfrak{s},\mathfrak{r})\, \| u \|^{2}_{\mathbf{L}^2   }
                                    \| u \|^{2}_{\mathbf{L}^\mathfrak{r}  }+
 c (k',\mathfrak{s},\mathfrak{r})\, \| u \|^{2}_{\mathbf{L}^2   }
\end{equation*}  
with positive constants depending on the parameters in parentheses and not necessarily the same in diverse applications, the constants being independent of $u$.
\end{lemma}

\begin{proof}  First, we note that under the hypothesis of the lemma, 
$u \in L^q ({\mathbb R}^n)$ for each $q\in [2,+\infty]$.
On using  the Leibniz rule and   the H\"older inequality we deduce that
\begin{equation}
\label{eq.En.Est3bu}
   \| (-\Delta)^{\frac{k'}{2}} \mathbf{D}  u \|^2_{\mathbf{L}^2   }
 \leq
   \sum_{j=0}^{k'}
   C^k_j\,
   \| \nabla^{k'+1-j} u \|^2_{\mathbf{L}^{\frac{2q}{q-1}}   }
   \| \nabla^{j} u \|^2_{\mathbf{L}^{2 q}    }
\end{equation}
with binomial type coefficients $C_{k'}^j$ and any $q \in (1,\infty)$.

For $k' = 0$ there are no other summands than that with $j=0$.
But for $k' \geq 1$ we have to consider the items corresponding to $1 \leq j \leq k'$, too.
The standard interpolation inequalities on compact manifolds
   (see for instance \cite[Theorem 2.2.1]{Ham82})
hint us that those summands which correspond to $1 \leq j \leq k'$ could actually be 
estimated by the item with $j=0$. We realize this as follows:
for any $j$ satisfying $1 \leq j \leq k'$ there are numbers
   $q > 1$ and
   $c > 0$
depending on $k'$ and $j$ but not on $u$, such that
\begin{equation}
\label{eq.En.Est3y}
   \| \nabla^{k'+1-j} u \|_{\mathbf{L}^{\frac{2 q}{q-1}}   }
   \| \nabla^{j} u \|_{\mathbf{L}^{2 q}   }
 \leq
   c \Big( \| \nabla^{k'+1} u \|_{\mathbf{L}^{\frac{2 \mathfrak{r}}{\mathfrak{r}-2}}   }
              \| u \|_{\mathbf{L}^{\mathfrak{r}}   }
         + \| u \|_{\mathbf{L}^{2}   }
     \Big).
\end{equation}
Indeed, we may apply Gagliardo-Nirenberg inequality \eqref{eq.L-G-N} if we prove that for each
$1 \leq j \leq k'$ there is a $q > 1$ depending on $k'$ and $j$, such that the system of algebraic
equations
\begin{equation*}
\left\{
\begin{array}{rcl}
   \displaystyle
   \frac{1}{2q}
 & =
 & \displaystyle
   \frac{j}{n}
 + \Big( \frac{\mathfrak{r}-2}{2 \mathfrak{r}}-\frac{k'+1}{n} \Big)\vartheta_1
 + \frac{1-\vartheta_1}{\mathfrak{r}},
\\
   \displaystyle
   \frac{q-1}{2q}
 & =
 & \displaystyle
   \frac{k'+1-j}{n}
 + \Big( \frac{\mathfrak{r}-2}{2 \mathfrak{r}} -\frac{k'+1}{n} \Big) \vartheta_2
 + \frac{1-\vartheta_2}{\mathfrak{r}}
\end{array}
\right.
\end{equation*}
admits solutions 
$
   \vartheta_1
 \in
   [\frac{j}{k'+1},1)$, 
$   \vartheta_2
 \in
   [\frac{k'+1-j}{k'+1},1)
$. 
On adding these equations we see that
$$
   \frac{1}{2} - \frac{k'+1}{n} - \frac{2}{\mathfrak{r}} = \Big( \frac{1}{2} - 
	\frac{k'+1}{n} - \frac{2}{\mathfrak{r}} \Big) (\vartheta_1 + \vartheta_2),
$$
i.e., the system is reduced to
$$
\left\{
\begin{array}{rcl}
   \vartheta_1 q (2 (k'+1) \mathfrak{r }+ 4n  - n \mathfrak{r})
 & =
 & 2 j\, \mathfrak{r} q + 2n q - n \mathfrak{r},
\\
   \vartheta_1 + \vartheta_2
 & =
 & 1.
\end{array}
\right.
$$
Choose
$
   \displaystyle
   \vartheta_1 = \frac{j}{k'+1}
$, 
$
   \displaystyle
   \vartheta_2 = \frac{k'+1-j}{k'+1}
$
to obtain 
$
   q
 = q (k',j)
 = \frac{(k'+1) \mathfrak{r}}{2 (k'+1) + j (\mathfrak{r} -4)} 
$.
Since
   $\mathfrak{r} > n \geq 3$ and
   $1 \leq j \leq k'$,
an easy calculation shows that
$$
\begin{array}{rcccccc}
   2 (k'+1) + j (\mathfrak{r} - 4)
 & >
 & 2 (k'+1) - 2j
 & \geq
 & 2
 & >
 & 0,
\\\
   (k'+1) \mathfrak{r} - (2 (k'+1) + j (\mathfrak{r} -4))
 & =
 & (k'+1) (\mathfrak{r}-2) - j (\mathfrak{r} - 4)
 & >
 & 0,
 &
 &
\end{array}
$$
i.e., $q (k',j)> 1$ in this case, and so \eqref{eq.En.Est3y} holds true. 
Therefore, if we choose $q (k',0) = \mathfrak{r} / 2 > 1$,  estimates of
   \eqref{eq.En.Est3bu} and
   \eqref{eq.En.Est3y}
readily yield
\begin{equation}
\label{eq.En.Est4b}
   \| (- \Delta)^{\frac{k'}{2}} \mathbf{D}  u \|^2_{\mathbf{L}^2   }
 \leq
   c (k',\mathfrak{r})
   \Big( \| \nabla^{k'+1} u \|^2_{\mathbf{L}^{\frac{2 \mathfrak{r}}{\mathfrak{r}-2}}   }
         \| u \|^2_{\mathbf{L}^{\mathfrak{r}}   }
       + \| u \|^2_{\mathbf{L}^{2}   }
   \Big)
\end{equation}
with a constant $c (k',\mathfrak{r})$ independent on $u$.

Now, if
   $\mathfrak{s} = 2$ and
   $\mathfrak{r} = +\infty$,
then, obviously, we get
\begin{equation}
\label{eq.En.Est3uuu}
   c (k',\mathfrak{r})\,
   \| \nabla^{k'+1} u \|^2_{\mathbf{L}^{\frac{2 \mathfrak{r}}{\mathfrak{r}-2}}   }
   \| u \|^2_{\mathbf{L}^{\mathfrak{r}}   }
 = c (k',\mathfrak{r})\,
   \| \nabla^{k'+1} u \|^2_{\mathbf{L}^{2}   }
   \| u \|^2_{\mathbf{L}^{\infty}   }.
\end{equation}
If
   $\mathfrak{s} > 2$ and
   $n < \mathfrak{r} < \infty$,
then we may again apply inequality \eqref{eq.L-G-N} with
   $j_0 = 0$,
   $k_0 = 1$,
   $q_0 = r_0 = 2$,
   $0 < a = n / \mathfrak{r} < 1$
and
   $p_0 = 2 \mathfrak{r}/(\mathfrak{r}-2)$
to achieve
\begin{equation}
\label{eq.L-G-N.gamma}
   \| \nabla^{k'+1} u \|_{\mathbf{L}^{\frac{2 \mathfrak{r}}{\mathfrak{r}-2}}   }
   \| u \|_{\mathbf{L}^{\mathfrak{r}}   }
 \leq
   c (\mathfrak{r})
   \Big( \| \nabla^{k'+2} u \|^{\frac{n}{\mathfrak{r}}}_{\mathbf{L}^2   }\,
   \| \nabla^{k'+1} u \|^{\frac{\mathfrak{r}-n}{\mathfrak{r}}}_{\mathbf{L}^2   }
 + \| u \|_{\mathbf{L}^{2}   }
   \Big)
   \| u \|_{\mathbf{L}^{\mathfrak{r}}   }
\end{equation}
with an appropriate Gagliardo-Nirenberg constant $c (\mathfrak{r})$ independent of $u$.

Since
$
   \displaystyle
   \mathfrak{s} = \frac{2 \mathfrak{r}}{\mathfrak{r}-n},
$
it follows from \eqref{eq.L-G-N.gamma} that
\begin{eqnarray}
\label{eq.En.Est3bbbb}
\lefteqn {
   c (k',\mathfrak{r})
   \| \nabla^{k'+1} u \|^2_{\mathbf{L}^{\frac{2 \mathfrak{r}}{\mathfrak{r}-2}}   }
   \| u \|^2_{\mathbf{L}^{\mathfrak{r}}   }
}
\nonumber
\\
 & \leq &
   2 c (k',\mathfrak{r})
   \Big( \| \nabla^{k'+2} u \|^{\frac{2n}{\mathfrak{r}}}_{\mathbf{L}^2   }\,
         \| \nabla^{k'+1} u \|^{\frac{2 (\mathfrak{r}-n)}{\mathfrak{r}}}_{\mathbf{L}^2   }\,
         \| u \|^2_{\mathbf{L}^{\mathfrak{r}}   }
       + \| u \|^2_{\mathbf{L}^{2}   }
         \| u \|^2_{\mathbf{L}^{\mathfrak{r}}    }
   \Big)
\nonumber
\\
 & \leq &
   \varepsilon\, \| \nabla^{k'+2} u \|^2_{\mathbf{L}^2   }
 + \frac{c (k',\mathfrak{r})}{\varepsilon} \| \nabla^{k'+1} u \|^{2}_{\mathbf{L}^2   }\,
                                          \| u \|^\mathfrak{s}_{\mathbf{L}^{\mathfrak{r}}    }
 + 2 c (k',\mathfrak{r})\, \| u \|^2_{\mathbf{L}^{2}   }
                          \| u \|^{2}_{\mathbf{L}^{\mathfrak{r}}    }
\end{eqnarray}
with some positive constants independent of $u$ because of Young's inequality 
\eqref{eq.Young} applied with    $p_1 = \mathfrak{r} / n$  and  
$p_2 = \mathfrak{r} / (\mathfrak{r}-n)$. Now, inequalities
   \eqref{eq.En.Est4b},
   \eqref{eq.En.Est3uuu} and
   \eqref{eq.En.Est3bbbb}
imply \eqref{eq.En.Est3bbb} for all
   $n < \mathfrak{r} \leq \infty$ and
   $2 \leq \mathfrak{s} = 2 \mathfrak{r} / (\mathfrak{r}-n) < \infty$,
as desired.
\end{proof}

We now introduce for $k \geq 1$ the following seminorm:
\begin{equation*}
   \| (f,u_0) \|_{k,\mu,T}
 = \Big( \| \nabla^k u_0 \|^2_{\mathbf{L}^2   }
       + 4 \mu ^{-1} \| \nabla^{k-1} f \|^2_{L^2 (I, \mathbf{L}^2   )} \Big)^{1/2}.
\end{equation*}

\begin{lemma}
\label{t.En.Est.g.2}
Let
   $k \in \mathbb{Z}_+$, $k+2>n/2$, and the pair 
   $\mathfrak{s}$, $\mathfrak{r}$ satisfy \eqref{eq.s.r}.
If
   $(u,p) \in B^{k,2,1}_\mathrm{vel,a} (I) \times B^{k+1,0,0}_\mathrm{pre,a} (I)$
is a solution to the Navier-Stokes type equations \eqref{eq.NS.map}  
corresponding to data
   $(f,u_0)$ in $B^{k,0,0}_\mathrm{for} (I) \times {\mathcal H}^{(a)}_{k+2}$
then
\begin{equation}
\label{eq.En.Est4}
   \| u \|_{j+1,\mu,T}
  \leq 
   c_j ( (f,u_0), u), \,\, 
   \| \nabla^{j} \mathbf{D} u \|_{L^2 (I, \mathbf{L}^2   )}
  \leq 
   c_j ( (f,u_0), u),
\end{equation}
$$
   \| \nabla^{j} \partial_t u \|^2_{L^2 (I, \mathbf{L}^2   )}
 + a^2 \| \nabla^{j+1} p \|^2_{L^2 (I, \mathbf{L}^2   )}
  \leq    c_j ( (f,u_0), u),
$$
for all $0 \leq j \leq k+1$, where the constants on the right-hand side depend on the norms 
   $\| (f,u_0) \|_{0,\mu,T}$,
   $\| (f,u_0) \|_{j+1,\mu,T}$ and
   $\| u \|_{L^\mathfrak{s} (I, \mathbf{L}^\mathfrak{r}   )}$
and need not be the same in diverse applications.
\end{lemma}

It is worth pointing out that the constants on the right-hand side of \eqref{eq.En.Est4} may  
also depend on $\mathfrak{s}$, $\mathfrak{r}$, $T$, $\mu$, etc., but we do not display this 
dependence in notation.

\begin{proof}
We first recall that
   $u \in C (I, \mathbf{H}^{k+2}   ) \cap L^2 (I, \mathbf{H}^{k+3}   )$,
   $u_0 \in \mathbf{H}^{k+2}   $
and
   $\nabla p, f \in C (I, \mathbf{H}^k   ) \cap L^2 (I, \mathbf{H}^{k+1}   )$
under the hypotheses of the lemma. 
Next, we see that in the sense of distributions we have
\begin{equation}
\label{eq.du.solution}
\left\{
\begin{array}{rclcl}
   (- \Delta)^{\frac{j}{2}}
   \Big(\partial_t - \mu \Delta  u 
	+ \mathbf{D} u + a\, \nabla p \Big)
 & =
 & (- \Delta)^{\frac{j}{2}} f
 & \mbox{in}
 & \mathbb{T}^n \times (0,T),
\\
   (- \Delta)^{\frac{j}{2}} u (x,0)
 & =
 & (- \Delta)^{\frac{j}{2}} u_0 (x)
 & \mbox{for}
 & x \in \mathbb{T}^n
\end{array}
\right.
\end{equation}
 for all $0 \leq j \leq k+1$, if $(u,p)$ 
is a solution to \eqref{eq.NS.map}.

Integration by parts and Remark \ref{r.Delta.r} yield
\begin{equation}
\label{eq.by.parts.3}
 -  ( (- \Delta)^{\frac{j+2}{2}} u, (- \Delta)^{\frac{j}{2}} u)_{\mathbf{L}^2   }
 = \| (- \Delta)^{\frac{j+1}{2}} u \|^2_{\mathbf{L}^2   }
 = \| \nabla^{j+1} u \|^2_{\mathbf{L}^2   },
\end{equation}
and similarly
\begin{equation}
\label{eq.dt.k}
   2\, (\partial_t (- \Delta)^{\frac{j}{2}} u, (- \Delta)^{\frac{j+2}{2}} u)_{\mathbf{L}^2   }
 = \frac{d}{dt}\, \| \nabla^{j+1} u \|^2_{\mathbf{L}^2   },
\end{equation}
cf. Lemma \ref{l.Lions}. Furthermore, as
   $\mathrm{rot}\, \nabla u = 0$ and
   $a \, \mathrm{div}\, u =0$
in ${\mathbb T}^n \times [0,T]$, we conclude that for all $t \in [0,T]$ 
\begin{equation*}
  a\, ((- \Delta)^{\frac{j}{2}} \nabla p (\cdot,t), (- \Delta)^{\frac{j+2}{2}} u (\cdot,t))_{\mathbf{L}^2   }
 = 
	\end{equation*}
$$
  a\, \lim_{i \to \infty}
   ((- \Delta)^{\frac{j}{2}} \nabla p_i (\cdot,t),
    (\mathrm{rot})^\ast \mathrm{rot}\, (- \Delta)^{\frac{j}{2}} u (\cdot,t))_{\mathbf{L}^2   }
  = 
$$
	\begin{equation}
\label{eq.by.parts.3p}
  a\, \lim_{i \to \infty}
   ((- \Delta)^{\frac{j}{2}} \mathrm{rot}\, \nabla p_i (\cdot,t),
    \mathrm{rot}\, (- \Delta)^{\frac{j}{2}} u (\cdot,t))_{\mathbf{L}^2   } =0,
\end{equation}
 where
   $p_i (\cdot,t) \in H^{j+2} $
is any sequence approximating $p (\cdot, t)$ in $H^{j+1} $.

On combining
   \eqref{eq.du.solution},
   \eqref{eq.by.parts.3},
   \eqref{eq.dt.k} and
   \eqref{eq.by.parts.3p}
we get
\begin{equation}
\label{eq.by.parts.22}
   2\,
   ( (- \Delta)^{\frac{j}{2}} (\partial_t - \mu \Delta u 
	+ \mathbf{D} u + a\, \nabla p) (\cdot,t),
     (- \Delta)^{\frac{j+2}{2}} u (\cdot, t)
   )_{\mathbf{L}^2   } =
\end{equation}
\begin{equation*}
   \frac{d}{dt} \| \nabla^{j\!+\!1} u (\cdot,t) \|^2_{\mathbf{L}^2   } 
 + 2 \mu \| \nabla^{j\!+\!2} u (\cdot,t) \|^2_{\mathbf{L}^2   } +
 2 ( (- \Delta)^{\frac{j}{2}} \mathbf{D} u (\cdot,t), (- \Delta)^{\frac{j\!+\!2}{2}} u 
(\cdot, t) )_{\mathbf{L}^2   } 
\end{equation*}
for all $0 \leq j \leq k+1$.
Next, according to the H\"older inequality, we get
\begin{equation}
\label{eq.En.Est3b}
   2 |( (- \Delta)^{\frac{j}{2}} \mathbf{D}  u, (- \Delta)^{\frac{j+2}{2}} u )_{\mathbf{L}^2   }|
 \leq
   \frac{2}{\mu}\, \| (- \Delta)^{\frac{j}{2}} \mathbf{D} u \|^2_{\mathbf{L}^2   }
 + \frac{\mu}{2}\, \| (- \Delta)^{\frac{j+2}{2}} u (\cdot,t) \|_{\mathbf{L}^2   },
\end{equation}
$$
   2\, ( (- \Delta)^{\frac{j}{2}} f (\cdot,t), (- \Delta)^{\frac{j+2}{2}} u (\cdot,t)
       )_{\mathbf{L}^{2}   }
 \leq 
   2\,  \| (- \Delta)^{\frac{j}{2}} f (\cdot,t) \|_{\mathbf{L}^{2}   }
        \| (- \Delta)^{\frac{j+2}{2}} u (\cdot,t) \|_{\mathbf{L}^{2}   }
\leq 
$$
\begin{equation}
\label{eq.En.Est3a}
   \frac{4}{\mu}\, \| (- \Delta)^{\frac{j}{2}} f (\cdot,t) \|^2_{\mathbf{L}^{2}   }
 + \frac{\mu}{4}\, \| (- \Delta)^{\frac{j+2}{2}} u (\cdot,t) \|^2_{\mathbf{L}^{2}   }
\end{equation}
for all $t \in [0,T]$.
By the H\"older inequality with
$
   \displaystyle
   q_1 = \frac{\mathfrak{r}}{n}
$
and
$
   \displaystyle
   q_2 = \frac{\mathfrak{r}}{\mathfrak{r}-n},
$
\begin{equation}
\label{eq.En.Est3x00}
   \int_0^t
   \| u (\cdot,s) \|^{2}_{\mathbf{L}^{2}   }
   \| u (\cdot,s) \|^{2}_{\mathbf{L}^{\mathfrak{r}}   }
   ds
 \leq
   \| u \|^{2}_{L^{\frac{2 \mathfrak{r}}{n} } ([0,t], \mathbf{L}^2   )}
   \| u \|^{2}_{L^\mathfrak{s} ([0,t], \mathbf{L}^{\mathfrak{r}}   )}.
\end{equation}
On summarising inequalities
   \eqref{eq.du.solution},
   \eqref{eq.by.parts.22},
   \eqref{eq.En.Est3b},
   \eqref{eq.En.Est3bbb},
   \eqref{eq.En.Est3a}   and 
\eqref{eq.En.Est3x00} 
we immediately obtain
\begin{eqnarray}
\label{eq.En.Est3x0}
\lefteqn{
   \| \nabla^{j+1} u (\cdot,t) \|^2_{\mathbf{L}^2   }
 + \mu \int_0^{t} \| \nabla^{j+2} u (\cdot,s) \|^2_{\mathbf{L}^{2}   } ds
}
\nonumber
\\
 & \leq &
   \| \nabla^{j+1} u_0 \|^2_{\mathbf{L}^{2}   }
 + \frac{4}{\mu} \| \nabla^{j} f \|^2_{L^2 (I,\mathbf{L}^2   )}
 + c (j, \mathfrak{s}, \mathfrak{r})
   \| u \|^{2}_{L^{\frac{2 \mathfrak{r}}{n}} ([0,t], \mathbf{L}^2   )}
   \| u \|^{2}_{L^\mathfrak{s} ([0,t],\mathbf{L}^{\mathfrak{r}}   )}
\nonumber
\\
 & + &
   c (j, \mathfrak{s}, \mathfrak{r})
   \frac{1}{\mu}
   \int_0^t \| u (\cdot,s) \|^{\mathfrak{s}}_{\mathbf{L}^{\mathfrak{r}}   }
            \| \nabla^{j+1} u (\cdot,s) \|^2_{\mathbf{L}^2   } ds
 + c (j, \mathfrak{s}, \mathfrak{r})
   \| u \|^{2}_{\mathbf{L}^{2}   }
\end{eqnarray}
for all $t \in [0,T]$.
It is worth to be mentioned that the constants need not be the same in diverse applications.
By
   \eqref{eq.En.Est2},
   \eqref{eq.En.Est2add}
and
   \eqref{eq.En.Est3x0},
given any $0 \leq j \leq k+1$, we get the following estimate for all $t \in I$: 
\begin{equation}
\label{eq.En.Est3x}
   \| \nabla^{j+1} u (\cdot,t) \|^2_{\mathbf{L}^2   }
 + \mu \int_0^{t} \| \nabla^{j+2} u (\cdot,s) \|^2_{\mathbf{L}^{2}   } ds
 \leq 
\end{equation}
\begin{equation*}
   \| (f,u_0) \|^2_{j+1,\mu,T}
 + c (j, \mathfrak{s}, \mathfrak{r}) T^{\frac{3}{\mathfrak{r}}}
   \| (f,u_0) \|^2_{0,\mu,T}
   \| u \|^{2}_{L^\mathfrak{s} ([0,t],\mathbf{L}^{\mathfrak{r}}   )}
 + 
\end{equation*}
\begin{equation*}
   c (j, \mathfrak{s}, \mathfrak{r}) \frac{1}{\mu}
   \int_0^t \| u (\cdot,s) \|^{\mathfrak{s}}_{\mathbf{L}^{\mathfrak{r}}   }
            \| \nabla^{j+1} u (\cdot,s) \|^2_{\mathbf{L}^2    } ds
 + c (j, \mathfrak{s}, \mathfrak{r}) T\, \| (f,u_0) \|^2_{0,\mu,T}.
\end{equation*}

On applying Gronwall's type Lemma \ref{l.Perov} 
to \eqref{eq.En.Est3x} with $\gamma_0 =1$, 
\begin{equation*}
   \mathfrak{A}   = 
   \| (f,u_0) \|^2_{j+1,\mu,T}
 + \left( c (j, \mathfrak{s}, \mathfrak{r}) T^{\frac{3}{\mathfrak{r}}}
          \| u \|^{2}_{L^\mathfrak{s} ([0,t],\mathbf{L}^{\mathfrak{r}}   )}
        + c (j, \mathfrak{s}, \mathfrak{r}) T
   \right)
   \| (f,u_0) \|^2_{0,\mu,T},
\end{equation*}
\begin{equation*}   \mathfrak{F} (t)
  =   \| \nabla^{j+1} u (\cdot,t) \|^2_{\mathbf{L}^2   },
\,\, 
   B (t)
  = 
   c (j, \mathfrak{s}, \mathfrak{r}) \frac{1}{\mu}
   \| u (\cdot,t) \|^{\mathfrak{s}}_{\mathbf{L}^{\mathfrak{r}}   }
\end{equation*}
we conclude that,
   for all $t \in [0,T]$ and $0 \leq j \leq k+1$,
\begin{equation}
\label{eq.d.sup}
   \| \nabla^{j+1} u (\cdot,t) \|^2_{\mathbf{L}^2   }
 \leq
   c (j, \mathfrak{s}, \mathfrak{r}, T, \mu, (f,u_0))
   \exp
   \Big(
   c (j, \mathfrak{s}, \mathfrak{r}) \frac{1}{\mu}
   \int_0^t
   \| u (\cdot,s) \|^{\mathfrak{s}}_{\mathbf{L}^{\mathfrak{r}}   }
   ds
   \Big)
\end{equation}
with a positive constant $c (j, \mathfrak{s}, \mathfrak{r}, T, \mu, (f,u_0))$ independent of $u$.
Obviously, \eqref{eq.En.Est3x} and \eqref{eq.d.sup} imply the first estimate of \eqref{eq.En.Est4}.

Next, applying \eqref{eq.En.Est3bbb} and \eqref{eq.En.Est3x00} we see that
\begin{equation*}
   \| (- \Delta)^{\frac{j}{2}} \mathbf{D} u \|^2_{L^2 ([0,t], \mathbf{L}^2   )}
 \leq 
\end{equation*}
\begin{equation*}
   \| \nabla^{j+2} u \|^2_{L^2 ([0,t], \mathbf{L}^2   )}
 + c (j, \mathfrak{s}, \mathfrak{r},\varepsilon\!=\!1)\,
   \| u \|^{\mathfrak{s}}_{L^\mathfrak{s} ([0,t],\mathbf{L}^{\mathfrak{r}}   )}
   \| \nabla^{j+1} u \|^2_{C ([0,t], \mathbf{L}^2   )}  + 
  \end{equation*}
	\begin{equation*}
    2 c (j, \mathfrak{r})\,\| u \|^{2}_{L^{\frac{2 \mathfrak{r}}{3}} ([0,t], \mathbf{L}^2   )}
   \| u \|^{2}_{L^\mathfrak{s} ([0,t], \mathbf{L}^{\mathfrak{r}}   )}
 + 2 c (j, \mathfrak{r})\, \| u \|^{2}_{L^{2} ([0,t], \mathbf{L}^2   )},
\end{equation*}
the constants being independent of $u$.
So, the second estimate of \eqref{eq.En.Est4} follows from
   \eqref{eq.En.Est2} and
   \eqref{eq.En.Est4}.

We are now ready to establish the desired estimates on $\partial_t u$ and $p$.
Indeed, since $a\, \mathrm{div}\, u = 0$, we get
\begin{equation}
\label{eq.dtu+dp}
   \| (- \Delta)^{\frac{j}{2}} (\partial_t u + a\, \nabla p) \|^2_{\mathbf{L}^{2}   }
 = \| \nabla^{j} \partial_t u \|^2_{\mathbf{L}^{2}   }
 +a^2 \, \| \nabla^{j+1} p \|^2_{\mathbf{L}^{2}   }
\end{equation}
for all $j$ satisfying $0 \leq j \leq k+1$.
From \eqref{eq.du.solution} it follows that
\begin{eqnarray}
\label{eq.p+u}
\lefteqn{
   \frac{1}{2}\,
   \| (- \Delta)^{\frac{j}{2}} (\partial_t u + a\, \nabla p) \|^2_{L^2 (I,\mathbf{L}^2   )}
}
\nonumber
\\
 & \leq &
   \| \nabla^{j} f \|^2_{L^2 (I,\mathbf{L}^2   )}
 + \mu\, \| \nabla^{j+2} u \|^2_{L^2 (I, \mathbf{L}^2   )}
 + \| (- \Delta)^{\frac{j}{2}} \mathbf{D} u \|^2_{L^2 (I, \mathbf{L}^2   )}
\end{eqnarray}
for all $0 \leq j \leq k+1$.
Therefore, the third estimate of \eqref{eq.En.Est4} follows from
   the first and second estimates of \eqref{eq.En.Est4},
   \eqref{eq.dtu+dp} and
   \eqref{eq.p+u},
showing the lemma.
\end{proof}

Clearly, we may obtain additional information on $\partial_t u$ and $p$.

\begin{lemma}
\label{c.En.Est.g.k1}
Under the hypotheses of Lemma \ref{t.En.Est.g.2},
\begin{equation}
\label{eq.En.EstD.C}
\begin{array}{rcl}
   \| \nabla^{j} \mathbf{D} u \|_{C (I, \mathbf{L}^2   )}
 & \leq &
   c_j ( (f,u_0), u),
\\
   \| \nabla^{j} \partial_t u \|^2_{C (I, \mathbf{L}^2   )}
 + a^2 \, \| \nabla^{j+1} p \|^2_{C (I, \mathbf{L}^2   )}
 & \leq &
   c_j ( (f,u_0), u)
\end{array}
\end{equation}
for all $0 \leq j \leq k$, with a positive constant $c_j ( (f,u_0), u)$ depending on the
norms
   $\| (f,u_0) \|_{0,\mu,T}, \ldots, \| (f,u_0) \|_{k+2,\mu,T}$,
   $\| \nabla^{j} f \|_{C (I, \mathbf{L}^{2}   )}$
and
   $\| u \|_{L^\mathfrak{s} (I, \mathbf{L}^\mathfrak{r}   )}$.
\end{lemma}

As mentioned, the constants on the right-hand side of \eqref{eq.En.EstD.C} may also depend
on $\mathfrak{s}$, $\mathfrak{r}$, $T$, 
$\mu$, etc., but we do not display this dependence
in notation.

\begin{proof}
Using \eqref{eq.du.solution}, we get
\begin{equation}
\label{eq.En.Est.41.tp+}
   \sup_{t \in [0,T]}
   \| (- \Delta)^{\frac{j}{2}} (\partial_t u + \nabla p) (\cdot, t) \|^2_{\mathbf{L}^{2}   }
 \leq 
   \sup_{t \in [0,T]}
   \| (- \Delta)^{\frac{j}{2}} (f + \mu \Delta u + \mathbf{D} u) (\cdot, t) \|^2_{\mathbf{L}^{2}   }
\leq 
\end{equation}
$$
   2
   \sup_{t \in [0,T]}
   \left( \| \nabla^{j} f (\cdot, t) \|^2_{\mathbf{L}^{2}   }
        + \| \nabla^{j+2} u (\cdot, t) \|^2_{\mathbf{L}^{2}   }
        + \| \nabla^{j} \mathbf{D} u (\cdot, t) \|^2_{\mathbf{L}^{2}   }
   \right)
$$
for all $0 \leq j \leq k$. The first two summands in the last line of \eqref{eq.En.Est.41.tp+} 
can be estimated via the data   $(f,u_0)$ and
   $\| u \|_{L^\mathfrak{s} (I,\mathbf{L}^\mathfrak{r}   )}$
using Lemma \ref{t.En.Est.g.2}.

On applying Lemma \ref{l.En.Est.Du.L2} to the third summand in \eqref{eq.En.Est.41.tp+} we 
see that
\begin{equation*}
   \| \nabla^j \mathbf{D} u \|^2_{C (I, \mathbf{L}^2   )}
\leq 
   \| \nabla^{j+2} u \|^2_{C (I, \mathbf{L}^2   )}
 + c (j, \mathfrak{s}, \mathfrak{r}, \varepsilon\!=\!1)\,
   \| u \|^{\mathfrak{s}}_{C (I, \mathbf{L}^{\mathfrak{r}}   )}
   \| \nabla^{j+1} u \|^2_{C (I, \mathbf{L}^2   )} + 
	\end{equation*}
\begin{equation}
\label{eq.Dk-2D}
    c (j, \mathfrak{s}, \mathfrak{r})\,
   \| u \|^{2}_{C (I, \mathbf{L}^2   )}
   \| u \|^{2}_{C (I, \mathbf{L}^{\mathfrak{r}}   )}
 + c (j, \mathfrak{s}, \mathfrak{r})\, \| u \|^{2}_{C (I, \mathbf{L}^2   )}
	\end{equation}
for all $0 \leq j \leq k$, the constants being independent of $u$. 
On the other hand, we may use the Sobolev embedding theorem (see for instance
   \cite[Ch.~4, Theorem 4.12]{Ad03} or
   \eqref{eq.Sob.index},
to conclude that for any $\lambda \in [0,1/2)$ there exists a constant $c (\lambda)$ independent of
$u$ and $t$,  such that
$$
    \| u (\cdot,t) \|_{\mathbf{C}^{0,\lambda}   }
 \leq
   c (\lambda)\,  \| u (\cdot,t) \|_{\mathbf{H}^{2+k}   }
$$
for all $t \in [0,T]$.
Then energy estimate \eqref{eq.En.Est2} and Lemma \ref{t.En.Est.g.2} imply immediately that
\begin{equation}
\label{eq.Sob.0}
   \sup_{t \in [0,T]} \|u (\cdot,t) \|_{\mathbf{C}^{0,\lambda}   }
 \leq
    c ((f,u_0), u),
\end{equation}
where the constant $c ((f,u_0), u)$ depends on
   $\| (f,u_0) \|_{j',\mu, T}$ with $j' = 0, \dots k+1$, 
and
   $\| u \|_{L^\mathfrak{s} (I, \mathbf{L}^\mathfrak{r}   )}$,
if inequality \eqref{eq.Sob.index} is fulfilled.
In particular,
\begin{equation}
\label{eq.CT.LsLr}
   \| u \|^{\mathfrak{s}}_{C (I, \mathbf{L}^{{\mathfrak r}}   )}
 \leq
   T \ell^{\frac{n \mathfrak{s}}{\mathfrak{r}}}
   \sup_{t \in [0,T]} \| u (\cdot,t) \|^{\mathfrak{s}}_{\mathbf{C}   }
 \leq
    T \ell^{\frac{n \mathfrak{s}}{\mathfrak{r}}}
    c ((f,u_0), u)
\end{equation}
with constant $c ((f,u_0), u)$ from \eqref{eq.Sob.0}. 
Hence, the first estimate of \eqref{eq.CT.LsLr} 
is fulfilled. 

At this point Lemma \ref{t.En.Est.g.2} and
   \eqref{eq.En.Est2add},
   \eqref{eq.En.Est.41.tp+},
   \eqref{eq.Dk-2D} and
   \eqref{eq.CT.LsLr}
allow us to conclude that
\begin{equation}
\label{eq.En.Est.41.tp++}
   \sup_{t \in [0,T]}
   \| (- \Delta)^{\frac{j}{2}} (\partial_t u + a\, \nabla p) (\cdot, t) \|^2_{\mathbf{L}^{2}   }
 \leq
   c (j, (f,u_0), u)
\end{equation}
for all $j = 0, 1, \ldots, k$, where $c (j, (f,u_0), u)$ is a positive constant depending on
   $\| (f,u_0) \|_{j',\mu,T}$ with $0 \leq j' \leq k+2$,
   $\| u \|_{L^\mathfrak{s} (I, \mathbf{L}^\mathfrak{r}   )}$
and
   $T$.
Hence, the second estimate of \eqref{eq.En.EstD.C} follows from
   \eqref{eq.dtu+dp} and
   \eqref{eq.En.Est.41.tp++}.
\end{proof}

Our next objective is to evaluate the higher derivatives of both $u$ and $p$ with respect to 
$x$ and $t$.

\begin{lemma}
\label{c.En.Est.g.ks}
Suppose that
   $s\in \mathbb{N}$, 
   $k \in \mathbb{Z}_+$, $2s+k>\frac{n}{2}$, 
and
   $\mathfrak{s}$, $\mathfrak{r}$ satisfy \eqref{eq.s.r}.
If
   $(u,p) \in B^{k,2s,s}_\mathrm{vel,a} (I) \times B^{k+1,2(s-1),s-1}_\mathrm{pre,a} (I)$
is a solution to the  Navier-Stokes type equations of \eqref{eq.NS.map}  
 with data
   $(f,u_0) \in B^{k,2(s-1),s-1}_\mathrm{for} (I) \times {\mathcal H}^{(a)}_{k+2s}$
then it is subjected to an estimate of the form
\begin{equation}
\label{eq.En.Est.Bks}
   \| (u,p) \|_{B^{k,2s,s}_\mathrm{vel,a} (I) \times B^{k+1,2(s-1),s-1}_\mathrm{pre,a} (I)}
 \leq
   c (k, s, (f,u_0), u),
\end{equation}
the constant on the right-hand side depending on
   $\| f \|_{B^{k,2(s-1),s-1}_\mathrm{for} (I)}$,
   $\| u_0 \|_{\mathbf{H}^{2s+k}}$
and
   $\| u \|_{L^\mathfrak{s} (I, \mathbf{L}^\mathfrak{r}   )}$
as well as on
   $\mathfrak{r}$, $\mathfrak{s}$, 	and $\mu$.
\end{lemma}

\begin{proof}
For $s=1$ and any $k \in \mathbb{Z}_+$, the statement of the lemma was proved in Lemmata
   \ref{t.En.Est.g.2} and
   \ref{c.En.Est.g.k1}.

Then the statement follows by induction with respect to $s$ from the recurrent formulas
\begin{equation}
\label{eq.recurrent}
\begin{array}{rcl}
   \partial^\alpha \partial^{j}_t (\partial_t u + a\, \nabla p)
 & =
 & \partial^\alpha \partial^{j}_t (f + \mu \Delta 
- \mathbf{D} u),
\\
   \| \partial^\alpha \partial^{j}_t (\partial_t u  + a\, \nabla p) \|^2_{\mathbf{L}^{2}   }
 & =
 & \| \partial^\alpha \partial^{j+1}_t u \|^2_{\mathbf{L}^{2}   }
 + a^2 \, \| \partial^\alpha \partial^{j}_t \nabla p \|^2_{\mathbf{L}^{2}   }
\end{array}
\end{equation}
provided that
   $a\, \mathrm{div}\, u = 0$
and
   $j \in \mathbb{Z}_+$,
   $\alpha \in \mathbb{Z}^n_+ $
are fit for the assumptions.

Indeed, suppose the assertion of the lemma is valid for $s = s_0$ and any $k \in \mathbb{Z}_+$ 
with $2s_0+k>\frac{n}{2}$. 
We then prove that it is fulfilled for $s = s_0+1$ and any $k \in {\mathbb Z}_+$ 
with $2s_0+k+2>\frac{n}{2}$. 
As
$
  (u,p)
  \in
  B^{k,2(s_0+1),s_0+1}_\mathrm{vel,a} (I) \times B^{k+1,2s_0,s_0}_\mathrm{pre,a} (I)$,
	$   (f,u_0)
  \in
  B^{k,2s_0,s_0}_\mathrm{for} (I) \times {\mathcal H}^{(a)}_{2(s_0+1)+k}
$, 
then, by Lemma \ref{l.emb.Bochner},
$
   (u,p)
  \in
  B^{k+2,2s_0,s_0}_\mathrm{vel,a} (I) \times B^{k+3,2(s_0-1),s_0-1}_\mathrm{pre,a} (I)$, 
$   (f,u_0)
  \in
  B^{k+2,2(s_0-1),s_0-1}_\mathrm{for} (I) \times V_{2s_0+(k+2)} 
$. 
Thus, $k+2>\frac{n}{2}-2s_0$ and, by the induction assumption,
\begin{equation}
\label{eq.En.Est.Bks.prime}
   \| (u,p) \|_{B^{k+2,2s_0,s_0}_\mathrm{vel,a} (I) \times B^{k+1,2(s_0-1),s_0-1}_\mathrm{pre,a} (I)}
 \leq
   c (k, s_0, (f,u_0), u),
\end{equation}
where the properties of the constant $c (k, s_0, (f,u_0), u)$ are similar to those described in
the statement of the lemma.

On the other hand, it follows from the first equality of \eqref{eq.recurrent} that for all 
suitable $j$ we get
\begin{equation*}
   \| \nabla^{j} \partial^{s_0}_t (\partial_t u +a\,  \nabla p) \|^2_{\mathbf{L}^{2}   }
 = 
   \| \nabla^{j} \partial^{s_0}_t (f + \mu \Delta 
	- \mathbf{D} u) \|^2_{\mathbf{L}^{2}   }
\leq 
\end{equation*}
\begin{equation}
\label{eq.est.ind.1}
   2
   \left( \| \nabla^{j} \partial^{s_0}_t f \|^2_{\mathbf{L}^{2}   }
 + \mu\, \| \nabla^{j+2} \partial^{s_0}_t u \|^2_{\mathbf{L}^{2}   }
 + \| \nabla^{j} \partial^{s_0}_t \mathbf{D} u \|^2_{\mathbf{L}^{2}   }
   \right).
\end{equation}
By the induction assumption, if
   $0 \leq j \leq k+1$ and
   $0 \leq i \leq k$,
then the norms
$
   \| \nabla^{j} \partial^{s_0}_t f \|^2_{L^2 (I, \mathbf{L}^{2}   )}
$
and
$
   \| \nabla^{i} \partial^{s_0}_t f \|^2_{C (I, \mathbf{L}^{2}   )}
$
are finite and
\begin{equation}
\label{eq.est.ind.2}
\begin{array}{rcl}
   \| \nabla^{j+2} \partial^{s_0}_t u \|^2_{L^2 (I, \mathbf{L}^{2}   )}
 & \leq
 & c\, \| u \|^2_{B^{k+2,2s_0,s_0}_\mathrm{vel,a} (I)},
\\
   \| \nabla^{i+2} \partial^{s_0}_t u \|^2_{C (I, \mathbf{L}^{2}   )}
 & \leq
 & c\, \| u \|^2_{B^{k+2,2s_0,s_0}_\mathrm{vel,a} (I)}
\end{array}
\end{equation}
with constants $c$ independent of $u$ and not necessarily the same in diverse applications.
Besides, \eqref{eq.B.pos.bound} with $w = u$ yield
\begin{equation}
\label{eq.est.ind.3}
\begin{array}{rcl}
   \| \nabla^{j} \partial^{s_0}_t \mathbf{D} u \|^2_{L^2 (I, \mathbf{L}^{2}   )}
 & \leq
 & c\, \| u \|^4_{B^{k+2,2s_0,s_0}_\mathrm{vel,a} (I)},
\\
   \| \nabla^{i} \partial^{s_0}_t \mathbf{D} u \|^2_{C (I, \mathbf{L}^{2}   )}
 & \leq
 & c\, \| u \|^4_{B^{k+2,2s_0,s_0}_\mathrm{vel,a} (I)}
\end{array}
\end{equation}
provided
   $0 \leq j \leq k+1$ and
   $0 \leq i \leq k$,
the constants being independent of $u$.

Finally, combining
   \eqref{eq.En.Est.Bks.prime},
   \eqref{eq.est.ind.1},
   \eqref{eq.est.ind.2},
   \eqref{eq.est.ind.3}
with the second equality of \eqref{eq.recurrent}, we conclude that
$$
   \| (u,p) \|_{B^{k,2(s_0+1),s_0+1}_\mathrm{vel,a} (I) \times B^{k+1,2s_0,s_0}_\mathrm{pre,a} (I)}
 \leq
   c (k, s_0+1, (f,u_0), u),
$$
where the constant on the right-hand side depends on
   $\| f \|_{B^{k,2s_0,s_0}_\mathrm{for} (I)}$,
   $\| u_0 \|_{\mathbf{H}^{2(s_0+1)+k}}$
and
   $\| u \|_{L^\mathfrak{s} (I, \mathbf{L}^\mathfrak{r}   )}$
as well as on
   $\mathfrak{r}$, $T$, $\mu$, etc.
This proves the lemma.
\end{proof}

Keeping in mind Corollary \ref{c.clopen}, we are now in a position to show that the range of 
mapping \eqref{eq.map.A} is closed  if given subset   $S = S_\mathrm{vel,a} \times 
S_\mathrm{pre,a}$ of the product
   $B^{k,2s,s}_\mathrm{vel,a} (I) \times B^{k+1,2(s-1),s-1}_\mathrm{pre,a} (I)$ 
such that the image $\mathcal{A}_a (S)$ is precompact in the space
   $B^{k,2(s-1),s-1}_\mathrm{for} (I) \times {\mathcal H}^{(a)}_{2s+k}$,
the set $S_\mathrm{vel,a}$ is bounded in the space
   $L^{\mathfrak{s}} (I,\mathbf{L}^{\mathfrak{r}}   )$ with 
	a pair $\mathfrak{s}$, $\mathfrak{r}$ satisfying \eqref{eq.s.r}.

Indeed, let a pair
    $(f,u_0) \in B^{k,2(s-1),s-1}_\mathrm{for} (I) \times {\mathcal H}^{(a)}_{2s+k}$
belong to the closure of the range of values of the mapping $\mathcal{A}_a$.
Then there is a sequence $\{ (u_i,p_i) \}$ in
    $B^{k,2s,s}_\mathrm{vel,a} (I) \times B^{k+1,2(s-1),s-1}_\mathrm{pre,a} (I)$
such that the sequence
   $\{ (f_i, u_{i,0}) = \mathcal{A}_a (u_i, p_i) \}$
converges to  $(f,u_0)$ in the space $B^{k,2(s-1),s-1}_\mathrm{for} (I) \times 
{\mathcal H}^{(a)}_{2s+k}$. Consider the set $S = \{ (u_i, p_i)\}$ from the inverse image of $\mathcal{A}_a$. 
As the image $\mathcal{A} (S) = \{ (f_i, u_{i,0}) \}$ is precompact in
   $B^{k,2(s-1),s-1}_\mathrm{for} (I) \times {\mathcal H}^{(a)}_{2s+k}$,
it follows from our assumption that the subset
   $S_\mathrm{vel,a} = \{ u_i \}$ of
   $B^{k,2s,s}_\mathrm{vel,a} (I)$
is bounded in the space  
$L^\mathfrak{s} (I, \mathbf{L}^{\mathfrak{r}}   )$. 

\begin{lemma}\label{l.closure}
Let a pair
    $(f,u_0) \in B^{k,2(s-1),s-1}_\mathrm{for} (I) \times {\mathcal H}^{(a)}_{2s+k}$
belong to the closure of the range of values of the mapping $\mathcal{A}_a$.
If the corresponding sequence $\{ u_i \}$ from 
    $B^{k,2s,s}_\mathrm{vel,a} (I) $ has a subsequence 
bounded in 
   $L^{\mathfrak{s}} (I,\mathbf{L}^{\mathfrak{r}}   )$ with 
$\mathfrak{s}$, $\mathfrak{r}$ satisfying \eqref{eq.s.r} then $(f,u_0)$  
belongs to the range of the mapping $\mathcal{A}_a$.
\end{lemma}

\begin{proof} 
We may replace the sequence by its bounded subsequence. Then, on 
applying Lemmata \ref{p.En.Est.u.strong} and \ref{c.En.Est.g.ks} we conclude that the 
sequence   $\{ (u_i, p_i) \}$
is bounded in the space
   $ B^{k,2s,s}_\mathrm{vel,a} (I) \times B^{k+1,2(s-1),s-1}_\mathrm{pre,a} (I)$.
By the definition of $B^{k,2s,s}_\mathrm{vel,a} (I)$, the sequence $\{ u_i \} $ is bounded in
   $C (I, \mathbf{H}^{k+2s}   )$ and
   $L^2 (I, \mathbf{H}^{k+2s+1}   )$,
and the partial derivatives $\{ \partial_t^j u_i \}$ in time with $1 \leq j \leq s$ are 
bounded in
   $C (I, \mathbf{H}^{k+2(s-j)}   )$ and
   $L^2 (I, \mathbf{H}^{k+2(s-j+1)}   )$.
Therefore, there is a subsequence $\{ u_{i_m} \}$ such that 1) 
the sequence    $\{ \partial^{\alpha+\beta}_x \partial_t^j u_{i_m} \}$
converges weakly in $L^2 (I, \mathbf{L}^{2}   )$ if 
   $|\alpha| + 2j \leq 2s$ and
   $|\beta| \leq k+1$; 2) the sequence
   $\{ \partial^{\alpha+\beta}_x \partial_t^j u_{i_m} \}$
converges weakly-$^\ast$ in $L^\infty (I, \mathbf{L}^{2}   )$ provided that
   $|\alpha| + 2j \leq 2s$ and   $|\beta| \leq k$.
	
It is clear that the limit $u$ of  $\{ u_{i_m} \}$ is a weak solution 
to \eqref{eq.NS}, i.e it satisfy 
\begin{equation}
\label{eq.NS.weak}
\left\{
   \begin{array}{rcl}
   \displaystyle
   \frac{d}{dt} (u,v)_{\mathbf{L}^2   }
 + \mu  (\nabla u, \nabla v)_{\mathbf{L}^2   } 
 & =
 & \langle f - \mathbf{D} (u), v \rangle,
\\
   u (\cdot,0)
 & =
 & u_0
   \end{array}
\right.
\end{equation}
for all $v \in  {\mathcal H}^{(a)}_{1}$. Moreover $u$ has the following properties: 
1) each derivative
   $\partial^{\alpha+\beta}_x \partial_t^j u$
belongs to $L^2 (I, {\mathcal H}^{(a)}_0)$ if 
   $|\alpha| + 2j \leq 2s$ and
   $|\beta| \leq k+1$; 2) each derivative
   $\partial^{\alpha+\beta}_x \partial_t^j u$
belongs to $L^\infty (I, {\mathcal H}^{(a)}_{0})$ provided that
   $|\alpha| + 2j \leq 2s$ and
   $|\beta| \leq k$.

Hence, it satisfies \eqref{eq.NS.weak.u} and, as we have seen in the proof 
of Theorem \ref{t.open.NS.short} such a solution is unique. In addition, if 
\begin{equation}
\label{eq.0js-1}
   0 \, \leq \, j \leq
  s-1,
\,\,    |\alpha| + 2j \leq  2s,
\,\,    |\beta|
  \leq
  k,
\end{equation}
then
   $\partial^{\alpha+\beta}_x \partial_t^j u \in  L^2 (I, {\mathcal H}^{(a)}_1)$ and
   $\partial^{\alpha+\beta}_x \partial_t^{j+1} u \in L^2 (I, ({\mathcal H}^{(a)}_1)')$.
Applying Lemma \ref{l.Lions} we readily conclude that
   $\partial^{\alpha+\beta}_x  \partial_t^j u \in C (I, {\mathcal H}^{(a)}_0)$
for all $j$ and $\alpha$, $\beta$ satisfying (\ref{eq.0js-1}). 
Hence  it follows that $u$ belongs to the space
   $B^{k+2,2(s-1),s-1}_\mathrm{vel,a} (I)$.
Moreover, using formula \eqref{eq.B.pos.bound}
with $w=u$ implies that the derivatives
   $\partial^{\alpha+\beta}_x \partial_t^j \mathbf{D} u$
belong to $C (I, \mathbf{L}^2   )$ for all $j$ and $\alpha$, $\beta$ which satisfy
inequalities (\ref{eq.0js-1}).

Denoting by $\mathbf{I}$ the identity operator $\mathbf{I}: \mathbf{L}^2 \to \mathbf{L}^2$
we set 
$$
\mathbf{P}_a = 
\left\{ 
\begin{array}{lll}
\mathbf{I}, & \mbox{ if } & a=0, \\
\mathbf{P} , & \mbox{ if } & a=1. \\
\end{array}
\right.
$$
By the construction,  the operator $\mathbf{P}_a:$ maps
   $C (I,\mathbf{L}^2   )$ continuously into
   $C (I,\mathbf{L}^2   )$.
Therefore, since $u$ is a solution to \eqref{eq.NS.weak} we deduce that
$$
   \partial^{\beta}_x \partial_t^s  u
 = \partial^{\beta}_x \partial_t^{s-1} \mu \Delta  u 
 - \partial^{\beta}_x \partial_t^{s-1} \mathbf{P}_a \mathbf{D} u
 + \partial^{\beta}_x \partial_t^{s-1} \mathbf{P}_a f
$$
belongs to $C (I,{\mathcal H}^{(a)}_0)$ for all multi-indices $\beta$ such that $|\beta| \leq k$.
In other words, $u$ lies in
   $B^{k,2s,s}_\mathrm{vel,a} (I)$. 
	
	If $a=1$ we still have to recover the `pressure' $p$. 
With this purpose, applying Proposition \ref{c.Sob.d}, 
we conclude that there is $p \in B^{k+1,2(s-1),s-1}_\mathrm{pre,1} (I)$
such that 
$$
\nabla p
 = (I - \mathbf{P}) (f - \mathbf{D} u),
$$
i.e. the pair $(u,p) \in B^{k,2s,s}_\mathrm{vel,1} (I)
\times B^{k+1,2(s-1),s-1}_\mathrm{pre,1} (I)$  is a solution to \eqref{eq.NS.map} 
related to the pair $(f,u_0)$.
\end{proof}

Thus, it follows from Lemma \ref{l.closure}
that the image of the mapping in \eqref{eq.map.A} is closed 
under the assumptions of the present theorem. 
Then the statement of the theorem related to the surjectivity of the mapping follows from
    Corollary \ref{c.clopen}.
\end{proof}

\section{Existence theorems } 
\label{s.NS.ET}

As it is well known, for $n=2$ the Navier-Stokes equations admits unique 
regular solutions for regular data, see, for instance, \cite[Ch. VI, \S 4, Theorem 3]{Lady70} 
or \cite[Ch. 3, \S 3, Theorems 3.2, 3.5, 3.6]{Tema79} or \cite[P. I, \S 3, 
Theorem 3.2]{Tema95} in the periodic 
setting.  As expected, this is the case for \eqref{eq.NS} in the function spaces 
under the considerations.

\begin{proposition}
\label{p.exist.Bohcher.n=2}
Let $n=2$,    $k \in \mathbb{Z}_+$,    $s \in \mathbb{N}$, 
	and  non-linearity \eqref{eq.non-linearity} satisfy \eqref{eq.trilinear}.
Then for each $T > 0$ mapping \eqref{eq.map.A} generated by the Navier-Stokes
 type equations \eqref{eq.NS} is a homeomorphism.
\end{proposition}

\begin{proof} As it is well known, applying 
Gagliardo-Nirenberg inequality \eqref{eq.L-G-N} 
with  $q_0=r_0=2$, $j_0=0$, $m_0=1$, $p_0=\frac{1}{r'}$, 
$a_0 = \frac{n(r'-2)}{2 \, r'}$ with the exceptional case where $n=2$, $r'=+\infty$ and 
$a_0=1$, we see that for any pair of positive real numbers $s'$, $r'$,   
\begin{equation} \label{eq.sigma.rho}
\frac{2}{s'} + \frac{n}{r'} =\frac{n}{2} \mbox{ with }
\left\{
\begin{array}{lll} 
2< r'<+\infty, &  2< s' <+ \infty, & n=2, \\
2< r'\leq \frac{2n}{n-2}, &  2\leq s' <+ \infty, & n\geq 3, \\
\end{array}
\right.
\end{equation}
there is a positive constant $c_{r',s'}$ independent on $T$ such that 
\begin{equation}  \label{eq.8.n}
\| u\|_{L^{s'}  (I,\mathbf{L}^{r'}({\mathbb R}^n))}
\leq c_{r',s'}\Big(  \|\nabla u\|^{\frac{2}{s'}} _{L^2(I,\mathbf{L^2})} 
\| u\|^{\frac{s'-2}{s'}} _{L^\infty(I,\mathbf{L^2})}  + 
\| u\| _{L^\infty(I,\mathbf{L^2})}\Big)
\end{equation}
for all $u \in L^2 (I, \mathbf{H}^1) \cap L^\infty (I, \mathbf{L}^2)$.

Thus, for $n=2$, Lemma \ref{p.En.Est.u.strong} and \eqref{eq.8.n} yield the desired 
\textit{universal} $L^{\mathfrak s}
 (I, \mathbf{L}^{\mathfrak r})$-estimate with numbers 
${\mathfrak s} = s'$, ${\mathfrak r} = r'$ satisfying both 
\eqref{eq.s.r} and \eqref{eq.sigma.rho}. 
Then the statement of the proposition follows from Theorem \ref{t.sr}. 
\end{proof}

However, for $n\geq 3$ inequalities \eqref{eq.s.r} and \eqref{eq.sigma.rho}
never coincide. Hence, we proceed with the standard result related to existence 
of regular solutions to Navier-Stokes type equations in higher dimensions 
on small intervals of the time variables, see, for instance, 
\cite[P. I, \S 3, Theorem 3.2]{Tema95} in the periodic setting.
   
\begin{proposition}
\label{p.exist.Bohcher.small.time}
Let $n\geq 3$,    $k \in \mathbb{Z}_+$,
   $s \in \mathbb{N}$, $2s+k> \frac{n}{2}$, 
	and  non-linearity \eqref{eq.non-linearity} satisfy \eqref{eq.trilinear}.
Then, given $(f,u_0) \in B^{k,2(s-1),s-1}_{\mathrm{for}} (I) \times 
\mathcal{H}^{(a)}_{2s+k}$  there is $T^* \in (0,T]$ and the unique solution 
$(u,p) \in B^{k,2s,s}_{\mathrm{vel,a}} ([0,T^*]) \times B^{k+1,2(s-1),s-1}_{\mathrm{pre,a}}
([0,T^*]) $ 
to the Navier-Stokes type equations \eqref{eq.NS.map}  with $T$ replaced by $T^*$, i.e.  
on ${\mathbb T}^n \times (0,T^*) $. 
\end{proposition}

We are now in a position to formulate and to prove existence theorems for regular solutions 
to \eqref{eq.NS} (more precisely, to \eqref{eq.NS.map}) in the case where $a=0$ and 
$n\geq 3$.  
 
\begin{theorem}
\label{t.exist.Bohcher}
Let $n\geq 3$, 
   $k \in \mathbb{Z}_+$,
   $s \in \mathbb{N}$, $2s+k> \frac{n}{2} +2$, 
	and  non-linearity \eqref{eq.non-linearity} satisfy \eqref{eq.trilinear}.
 If $a=0$ then for each $T > 0$ mapping \eqref{eq.map.A} generated by the Navier-Stokes
 type equations \eqref{eq.NS} is a homeomorphism.
\end{theorem}

\begin{proof} We  begin with a simple lemma.  
The fact that the torus is a compact closed manifold is essential in it because otherwise one 
should impose additional restrictions on the function $\eta$ under the consideration. 

\begin{lemma}
\label{l.diff.sigma}
Let    $k \in \mathbb{Z}_+$,
   $s \in \mathbb{N}$, $2s+k> \frac{n}{2}$, $a=0$ or $a=1$, and 
   $u \in B^{k,2s,s}_\mathrm{vel,a} (I)$. 
	Then for each non-negative function $\eta \in C^1[1,+\infty)$  
we have 
\begin{equation}
\label{eq.dt.eta}
   2 (\partial_t u (\cdot,t), (u \eta'(1+|u|^{2})) (\cdot,t))_{\mathbf{L}^2   }
 = \frac{d}{dt}\, \| \eta(1+|u|^2) (\cdot,t) \|_{L^{1}   };
\end{equation}
if, in addition, $\eta \in C^2[1,+\infty)$ 
 then on the interval $I = [0,T]$ we obtain
\begin{equation} \label{eq.Delta.eta}
 - (\Delta  u (\cdot,t), (u \eta' (1+|u|^{2})) (\cdot,t))_{\mathbf{L}^2   }\!
 = 
\end{equation}
\begin{equation*}
\int_{\mathcal Q}    \Big( |\nabla u|^2 \,  \eta'(1+|u|^{2})   
 + \frac{1}{2}  |\nabla |u |^{2}|^2 \, \eta'' (1+|u|^2)\Big) (\cdot,t) dx.
\end{equation*}
\end{lemma}

\begin{proof}
By the Sobolev embedding theorem (see \eqref{eq.Sob.index}), we get
   $u \in C (I,\mathbf{H}^{2s+k}   )$
and this latter space is embedded continuously into
   $C (I,\mathbf{C}^{0,\lambda}   )$
for all $0 \leq \lambda < 1/2$ if $2s+k> \frac{n}{2}$.
It follows that $u \in C(I, \mathbf{C}^{0,\lambda}   )$, 
  $|u|^2\in  C (I, \mathbf{C}^{0,\lambda}   ),$ 
    $\| \eta (1+|u|^2) (\cdot,t) \|_{L^{1}   } \in C(I)$ 
 and  $u \eta' (1+|u|^2) \in C (I, \mathbf{C}^{0,0}   )$.

Any continuous function on $(0,T)$ is a regular distribution there, and so it can be weakly
differentiated in the interval.
Thus,
$$
   \frac{d}{dt}\,  \left\| \eta(1+|u|^2) (\cdot,t) \right\|_{L^{1}   }
 \in \mathcal{D}' (0,1).
$$
On the other hand, as $u \in B^{k,2s,s}_\mathrm{vel,a} (I)$, then
   $\partial_t u \in C (I, \mathbf{L}^{2})$.
Hence it follows that 
\begin{equation*}
 - \left(  \| \eta(1+|u|^2) (\cdot,t) \|_{L^{1}   }, v' (t) \right)_{L^{2} (I)}
 = - \int_0^T \!\! \int_\mathcal{Q} \eta(1+|u|^2) (x,t) 
 dx  v' (t) dt =
\end{equation*}
\begin{equation*}
   2   \int_0^T \!\! \int_\mathcal{Q} \Big( \sum_{j=1}^n u^j \partial_t u^j \Big)
                            \eta'(1+|u|^2)  (x,t) dx\, v (t) dt
 = 
\end{equation*}
\begin{equation*}
  2 \Big( (\partial_t u (\cdot,t), (u \eta'(1+|u|^2)  (\cdot,t))_{\mathbf{L}^{2}   },
         v (t)
   \Big)_{L^{2} (I)}
\end{equation*}
for all smooth functions $v$ with compact support in $(0,T)$. 
Thus, \eqref{eq.dt.eta} holds true and the function
   $(\partial_t u, u  \eta'(1+|u|^2) )_{\mathbf{L}^{2}   }$
belongs to $C (I)$.

Next, as   $u \in  C (I, \mathbf{C}^{0,\lambda}   )\cap C (I, \mathbf{H}^{2s+k} ) $
for each $0 \leq \lambda < 1/2$, we see that 
\begin{eqnarray*}
\lefteqn{
   - (\Delta  u (\cdot,t), (u  \eta'(1+|u|^2)) (\cdot,t))_{\mathbf{L}^2   }
}
\\
 & = &
   \int_\mathcal{Q}
   \Big( \sum_{j,k=1}^n (\partial_j u^k)^2  \eta'(1+|u|^2)
       +  \sum_{j=1}^n  \frac{1}{2} (\partial_j |u|^{2}) \, \partial_j \eta'(1+|u|^2)
   \Big) (x,t)    dx = 
	\\
 & = &
   \int_\mathcal{Q}
   \Big( |\nabla u|^2  \eta'(1+|u|^2)
       +   \frac{1}{2} |\nabla |u|^{2}|^2 \,  \eta''(1+|u|^2)
   \Big) (x,t)    dx 
\end{eqnarray*}
for almost all $t \in [0,T]$ that gives precisely \eqref{eq.Delta.eta}. 
\end{proof}

If $\eta (z) = e^{ z }$,  $a=0$ or $a=1$, 
then for each $u \in B^{k,2s,s}_\mathrm{vel,a} (I)$ for almost all $t \in I$  
\begin{equation}
\label{eq.dt.exp}
(\partial_t u (\cdot,t), (u \eta' (1+|u|^{2})) (\cdot,t))_{\mathbf{L}^2   } 
= \frac{1}{2} \frac{d}{dt}\, \left\| e^{ 1+|u|^2} \right\|_{L^{1}   },
\end{equation}
\begin{equation} 
\label{eq.Delta.exp}
- (\Delta  u (\cdot,t), (u \eta'(1+|u|^2)) (\cdot,t))_{\mathbf{L}^2   }
 = 
\end{equation}
\begin{equation*}  
\int_{\mathcal Q} 
e^{1+|u|^2} \, \Big( |\nabla u|^2 + \frac{1}{2}| \nabla |u|^2|^2\Big) (\cdot,t)
dx = 
  \left\||\nabla u| e^{(1+|u|^2) /2} \right\|^2_{L^2   } 
+ 2 \left\|\nabla  e^{(1+|u|^2) /2} \right\|^2_{\mathbf{L}^2   }.
\end{equation*}

Now, if a pair $(u,p) \in B^{k,2s,s}_{\mathrm{vel,a}} (I) \times B^{k+1,2(s-1),s-1}
_{\mathrm{pre,a}} (I) $ is a solution to the Navier-Stokes type equations \eqref{eq.NS.map}  
corresponding to some data $(f,u_0) \in B^{k,2(s-1),s-1}_{\mathrm{for}} (I) \times 
\mathcal{H}^{(a)}_{2s+k}$, then formulas \eqref{eq.dt.exp}, \eqref{eq.Delta.exp} 
imply 
\begin{equation*}
     \frac{d}{dt} \left\| e^{ 1+|u|^2} (\cdot,t)\right\|_{L^{1}   }
 + 2\mu  \left\||\nabla u| e^{(1+|u|^2) /2} (\cdot,t)\right\|^2_{L^2   } 
+ 4 \mu \left\|\nabla  e^{(1+|u|^2) /2} (\cdot,t)\right\|^2_{\mathbf{L}^2   }  =
\end{equation*}
\begin{equation} \label{eq.non-lin.term}
 2 \Big(\mathbf{P}_a (f -  
\mathbf{D}u) (\cdot,t), u e^{(1+|u|^2) } (\cdot,t)\Big)_{\mathbf{L}^2   }
\end{equation}
for almost all $t \in [0,T]$  if $a=0$ or $a=1$ (because 
$a\, \mathbf{P}_a  \nabla p =0$). 

Let us transform  this energy identity to an a priori estimate.

First, as $n\geq 3$, we see that according to \eqref{eq.L-G-N}, 
\begin{equation} \label{l.G-L-N.exp}
\left\|\nabla  e^{(1+|u|^2) /2} (\cdot,t)\right\|^2_{\mathbf{L}^2   } \geq c^{(1)}_n 
\left\|  e^{(1+|u|^2) /2} (\cdot,t)\right\|^2_{L^{\frac{2n}{n-2}} } 
- c^{(2)}_n  \left\|  e^{(1+|u|^2) /2} (\cdot,t)\right\|^2_{L^{2}} 
\end{equation} 
for almost all $t \in [0,T]$ 
with the suitable Gagliardo-Nirenberg constants $c^{(j)}_n>0$ independent on $u$ and $t$.

Second, as $2s+k> \frac{n}{2} +2$, we see that $f , {\mathbf P}_a f 
\in C (I,{\mathbf C}^{0,\lambda})$ and then 
\begin{equation} \label{eq.f.exp}
2\left|({\mathbf P}_a f , u e^{(1+|u|^2)} )_{\mathbf{L}^2 }\right|\leq 
2\| {\mathbf P}_a f\|_{\mathbf{C}^{0,\lambda}}
\left\|  e^{(1+|u|^2)/2}\right\|_{L^2} 
\left\| (1+|u|^2)  e^{(1+|u|^2)/2}\right\|_{L^2}  \leq 
\end{equation}
$$
\|{\mathbf P}_a f \|^2_{\mathbf{C}^{0,\lambda}  } \left\| e^{(1+|u|^2)/2 }
\right\|^2_{L^2 }
+ \left\|(1+|u|^2) e^{(1+|u|^2)/2 }\right\|^{2}_{L^2} 
$$
for almost all $t \in [0,T]$. 

We continue with $a=0$ because we have no idea how to estimate the pseudo-differential 
term $ \mathbf{P} \mathbf{D}u$ in \eqref{eq.non-lin.term} for $a=1$.  

According to the H\"older inequality, see for instance \cite[Corollary 2.6]{Ad03},  
\eqref{eq.Young}, and  \eqref{eq.non-linearity} 
for all $u\in B^{k,2s,s}_{\mathrm{vel,a}} (I)$ the following inequality holds: 
\begin{equation} \label{eq.Du.exp}
2\left|(\mathbf{D}u , u e^{1+|u|^2 })_{\mathbf{L}^2   }\right|\leq 
\tilde C_{{\mathcal M},n} \left\| |\nabla u| e^{(1+|u|^2) /2} \right\|_{L^2} 
\left\| |u|^2  e^{(1+|u|^2)/2} \right\|_{L^2}  \leq 
\end{equation}
$$
\mu  \left\||\nabla u| e^{(1+|u|^2) /2} (\cdot,t)\right\|^2_{L^2   } + 
 C_{{\mathcal M},n} \mu^{-1}
\left\| (1+|u|^2)  e^{(1+|u|^2)/2}\right\|^2_{L^2}  
$$
for almost all $t \in [0,T]$ with positive constants $C_{{\mathcal M},n}$, 
$\tilde C^{(\mu)}_{{\mathcal M},n}$ independent on $u$. 

Therefore, combining \eqref{eq.non-lin.term}, \eqref{l.G-L-N.exp}, 
\eqref{eq.Du.exp}, \eqref{eq.f.exp},  
and applying an integration over the interval $(0,t)$,   we obtain 
\begin{equation} \label{eq.exp.int}
    \Big\| e^{1+|u|^2  } (\cdot,t)\Big\|_{L^{1}   }
 + \int_0^t\Big(\mu   \Big\||\nabla u| e^{\frac{1+|u|^2}{2}} \Big\|^2_{L^2   } 
+ 4\mu c^{(1)}_n  \Big\|e^{1+|u|^2 } \Big\|_{L^{\frac{n}{n-2}}   } \Big)d\tau \leq 
\end{equation}
\begin{equation*}
 \Big\| e^{1+|u_0|^2}  \Big\|_{L^{1}   } + 
 \int_0^t \Big( \Big( 4\mu c^{(2)}_n  + \|f\|^2_{\mathbf{C} ^{0,\lambda} } \Big)
\Big\| e^{ 1+|u|^2  }\Big\|_{L^{1}   } + 
\frac{C_{{\mathcal M},n}}{\mu} \Big\| (1+|u|^2)^2  e^{1+|u|^2 }
\Big\|_{L^1}  \Big)d\tau 
\end{equation*}
for all $t \in [0,T]$ with a constant $C^{(\mu)}_{{\mathcal M},n}>0$ independent on 
$f$, $u_0$ and $u$. 

According to Corollary \ref{c.clopen}, to complete the proof of the theorem, 
it is sufficient to prove that the range of mapping \eqref{eq.map.A} 
is closed. With this purpose,  let a pair
    $(f,u_0) \in B^{k,2(s-1),s-1}_\mathrm{for} (I) \times {\mathbf H}^{2s+k}$
belong to the closure of the range of values of the mapping $\mathcal{A}_0$.
 Then there is a sequence $\{ u_i \}$ in     $B^{k,2s,s}_\mathrm{vel,0} (I) $ 
such that the sequence
   $\{ (f_i, u_{0,i}) = \mathcal{A}_0 u_i \}$
converges to  $(f,u_0)$ in the space $B^{k,2(s-1),s-1}_\mathrm{for} (I) \times 
{\mathbf H}^{2s+k}$. In particular, for the sequences
$
A_i = \left\| e^{ (1+|u_{0,i}|^2)  }\right\|_{L^{1}   }$ ,  
$4\mu c^{(2)}_n  +\|f_i(\cdot, t)\|^2_{\mathbf{C}^{0,\lambda} }  
$, 
we have for all $i \in {\mathbb N}$ and   all $ t \in [0,T]$:  

\begin{equation} \label{eq.A}
A_i \leq \mathrm{Vol}({\mathbb T}^n) 
e^{ (1+C \, \|u_{0}\|^2_{\mathbf{C}})}   \leq 
\mathrm{Vol}({\mathbb T}^n) 
e^{ (1+C\, \|u_{0}\|^2_{\mathbf{H}^{2s+k}})}  =A=A(u_0)>0, 
\end{equation}
\begin{equation} \label{eq.B}
4\mu c^{(2)}_n  +\ \|f_i(\cdot, t)\|^2_{\mathbf{C}^{0,\lambda} }  \leq 
4\mu c^{(2)}_n +C\|f\|^2
_{B^{k,2(s-1),s-1}_\mathrm{for} (I)} =B=B(f)>0, 
\end{equation}
with a  positive constant $C$ because of the Sobolev embedding theorem. 

Thus, if $a=0$ then for all $i \in {\mathbb N}$ and all $t \in [0,T]$ we obtain 
\begin{equation*} 
    \left\| e^{ 1+|u_i|^2  } (\cdot,t)\right\|_{L^{1}   }
 + \int_0^t\Big(\mu  \left\||\nabla u_i| e^{(1+|u_i|^2) /2} \right\|^2_{L^2   } 
+ 4\mu c^{(1)}_n  \left\|e^{1+|u_i|^2 } \right\|_{L^{\frac{n}{n-2}}   } \Big)d\tau \leq 
\end{equation*}
\begin{equation} \label{eq.exp.int.i}
A +\int_0^t \!\! 
\Big( B\left\| e^{ 1+|u_i|^2  }\right\|_{L^{1}   }
 +   \mu^{-1}C_{{\mathcal M},n} \Big\| (1+|u_i|^2)^2  e^{1+|u_i|^2 }
\Big\|_{L^1} 
\Big) d\tau .
\end{equation}
If the pair $(f,u_0)$ does not belong to 
 the range of mapping \eqref{eq.map.A} then Lemma \ref{l.closure}
 implies that the sequence 
$\{ u_i \}$  is unbounded in the space $L^{{\mathfrak s}} (I, 
\mathbf{L}^{{\mathfrak r}})$ with any  ${\mathfrak s}$, ${\mathfrak r}$ 
satisfying \eqref{eq.s.r}. Assume that it is the case. Then, 
replacing, if necessary, 
the sequence $\{ u_i \} $ with its subsequence we may  consider that \eqref{eq.unbounded} 
holds with ${\mathfrak s}={\mathfrak r}=n+2$. On the other hand, $e^z > \frac{z^{n+2}}{(
n+2)!}$ for all positive $z$ and hence 
\begin{equation} \label{eq.s=r}
\int_{0}^T \Big\|e^{1+|u_i|^2}\Big\|  _{L^{1} } d\tau \geq 
\frac{\|1+|u_i|^2\| ^{n+2} _{L^{n+2}(I,L^{n+2} )}}{(n+2)!} \geq 
C_{n,T} \|u_i\| ^{2(n+2)} _{L^{n+2}(I,\mathbf{L}^{n+2}) }  
\end{equation}
because the set ${\mathbb T}^n \times [0,T]$ is compact. 
In particular, according to \eqref{eq.unbounded},
\begin{equation} \label{eq.unbounded.exp}
\lim_{i\to +\infty}\int_{0}^T \Big\|e^{1+|u_i|^2}\Big\|  _{L^{1} } d\tau =+\infty 
\end{equation}
Instead of trying to obtain a {universal} a priori estimate for the solution $u$, 
we will be matching various asymptotics related to the sequence $\{ u_i\}$ satisfying 
\eqref{eq.unbounded.exp}. 

Namely, since  
$$
\lim_{\delta \to +0}
\frac{T^{1/\delta}  \, (A+1) \,  e^{TB}
}{(\mu\, \ln{(1/\delta)})^{1/\delta}} =0,
$$
then there is a number $0<\delta_0 < 1/n$ 
such that for all $\delta \in (0, \delta_0]$ we have
\begin{equation} \label{eq.Gron.add.im} 
\frac{T^{1/\delta} \, (A+1) }{(\mu \, \ln{(1/\delta)})^{1/\delta}}< e^{-T B}
\end{equation}

With a clear intention to use Gronwall type Lemma \ref{l.Perov} 
we claim the following.

\begin{lemma} \label{l.claim} 
There is a number $\delta =\delta(f,u_0) \in (0,\delta_0]$ such that 
\begin{equation*}
 \mu^{-1}C_{{\mathcal M},n} \int_0^t \!\! \Big(
 \Big\| (1+|u_i|^2)^2  e^{1+|u_i|^2 }
\Big\|_{L^1} \Big) d\tau \leq 
\end{equation*}
\begin{equation*} 
 1+\int_0^t\Big( 
\frac{1}{\mu \, \delta \, \ln{(1/\delta)}}\left\|
 e^{ 1+|u_i|^2  }\right\|^{1+\delta}_{L^{1}   }+
 4\mu c^{(1)}_n  \left\|e^{1+|u_i|^2} \right\|_{L^{\frac{n}{n-2}}   } \Big)d\tau 
\end{equation*}
for all $i \in {\mathbb N}$ and all $t\in [0,T]$.
\end{lemma}

\begin{proof} 
We argue by contradiction. Passing to a discrete parameter $\delta=1/m$,  
assume that for any natural number $m\geq 1/\delta_0$  
there are $i_m \in {\mathbb N}$ and $t_m\in [0,T]$ such that 
\begin{equation} \label{eq.exp.int.im.not}
1+  \int_0^{t_m} \Big( \frac{m}{\mu \, \ln{(m)}}
  \int_0^{t_m} 	\left\| 
 e^{ 1+|u_{i_m}|^2  }\right\|^{\frac{m+1}{m}}_{L^{1}   } 
+ 4\mu c^{(1)}_n  \left\|e^{(1+|u_{i_m}|^2) } \right\|_{L^{\frac{n}{n-2}}   } \Big)d\tau <
\end{equation}
\begin{equation*}
 \mu^{-1}C_{{\mathcal M},n} 
\int_0^{t_m} \!\! \Big\| (1+|u_{i_m}|^2)^2  e^{1+|u_{i_m}|^2 }
\Big\|_{L^1}  d\tau .
\end{equation*}
Since the segment $[0,T]$ is a compact, then passing to a subsequence, 
we may assume that the sequence $\{t_m\}$
converges to a time $T_0 \in [0,T]$.

According to Proposition \ref{p.exist.Bohcher.small.time} we may assume also that 
the pair $(f,u_0)$ admits the unique solution $u $ to the Navier-Stokes type equations 
\eqref{eq.NS} on ${\mathbb T}^n \times [0,T)$ 
belonging to $B^{k,2s,s}_{\mathrm{vel,0}} ([0,T^*])  $  for any $T^* \in (0,T)$. 
In particular, the Open Mapping Theorem \ref{t.open.NS.short} and the Implicit 
Function Theorem for Banach spaces yield that the sequence
$\{ u_i \}$ converges to the solution $u$ in the space  $B^{k,2s,s}_\mathrm{vel,0} 
([0,T^*])$ for any $T^* \in (0,T)$.

If $T_0=0$ then passing to the limit  with respect to $m\to + \infty$ in  
\eqref{eq.exp.int.im.not} we obtain a contradiction: 
$
0<1 \leq 0
$. 

If $T_0 \in (0,T)$ then the solution $ u$ belongs to  $B^{k,2s,s}_\mathrm{vel,0} 
([0,T_0]) $.  As far as the space $B^{k,2s,s}_\mathrm{vel,0} ([0,T_0]) $ is 
embedded continuously into 
$C ([0,T_0] ,\mathbf{C}^2)$ for $2s+k>n/2+2$,  
passing to the limit with respect to $m\to + \infty$ in  
\eqref{eq.exp.int.im.not} we obtain a contradiction: 
$$
+\infty \leq 
\mu^{-1} C_{{\mathcal M},n} 
\int_0^{T_0} \!\! \Big\| (1+|u|^2)^2  e^{1+|u|^2 }
\Big\|_{L^1}  d\tau  < + \infty. 
$$
Therefore, the only possibility is that 
the sequence $\{t_m\}$ converges to $T$ and both the left and the right hand sides 
of \eqref{eq.exp.int.im.not} converge to $+\infty$. 

In particular, \eqref{eq.unbounded.exp} implies 
$$
\lim_{m\to +\infty} t_m^{\frac{1}{m+1}}
\left\|e^{1+|u_{i_m}|^2 } \right\|_{L^{\frac{m+1}{m}} ([0,t_m], L^{1})   } \geq 
\lim_{m\to +\infty}
\int_0^{t_m} \ \left\|e^{1+|u_{i_m}|^2 } \right\|_{L^{1}   } \, d\tau = +\infty.
$$
Hence,  passing, if necessary to a subsequence,  we may assume that
\begin{equation}\label{eq.asymp.1}
\left\|e^{1+|u_{i_m}|^2} \right\|_{L^{\frac{m+1}{m}} ([0,t_m], L^{1})   }
\geq m^{4m} \mbox{ for all } m\geq 1/\delta_0.
\end{equation}

Note  that  for $m'\in \mathbb N$ and $\gamma>0$ the function 
$z^{m'} e^{-\gamma z}$ is bounded on the semi-interval $[1,+\infty)$ and, in particular,  
\begin{equation*} 
|z^{m'} e^{-\gamma z}| \leq 
\Big(\frac{m'}{\gamma e}\Big)^{m'}   \mbox{ for all }
z \in [1,+\infty).
\end{equation*}

Then  we arrive at the following inequality with any $\gamma \in (0,1)$, 
$v \in B^{k,2s,s}_\mathrm{vel,0} (I) $: 
\begin{equation*} 
\int_0^t \Big\|(1+|v|^{2})^2 e^{1+|v|^2 }\Big\|_{L^1}
d\tau 
= 
\int_0^t \Big( \Big\|(1+|v|^2)^{2} e^{-(1+|v|^2)\gamma} 
e^{(1+|v|^2)(1+\gamma)}\Big\|_{L^1} 
\Big) d\tau
\leq 
\end{equation*}
\begin{equation*} 
\Big(\frac{2}{ \gamma e}\Big)^2 
\int_{0}^t \Big\|e^{1+|v|^2}\Big\|^{1+\gamma} _{L^{1+\gamma}} d\tau  . 
\end{equation*}

Next, we notice that if
   $1 \leq p' < p < p''$
then, according to the H\"older inequality with
   $q = 1 / \vartheta$ and
   $q' = 1 / (1-\vartheta)$,
we have an interpolation inequality
\begin{equation*}
   \| w\|_{L^{p} }
 = \Big( \int_{{\mathcal Q}} (|w|^{p''})^{\vartheta} |w|^{p-\vartheta p'')} dx 
\Big)^{\frac{1}{p}}
 \leq
   \| w \|^{\frac{p'' \vartheta}{p}}_{L^{p''}}\,
   \| w\|^{\frac{p' (1-\vartheta)}{p}}_{L^{p'}},
\end{equation*}
see for instance \cite[Theorem 2.11]{Ad03}, where 
$  \vartheta=   \frac{p-p'}{p''-p'}\in (0,1), \, 
   (1-\vartheta)=   \frac{p''-p}{p''-p'} \in (0,1) $. 
In particular, for $p=1+\gamma$, $p'=1$, $p''=\frac{n}{n-2}$ we obtain
$$\Big(\frac{2}{ \gamma e}\Big)^2
\Big\|e^{1+|v|^2}\Big\| ^{\gamma+1}_{L^{1+\gamma}} 
\leq 
\Big(\frac{2}{ \gamma e}\Big)^2 \Big\| e^{1+|v|^2} \Big\|^{\frac{\gamma n}{2 } }
_{L^{\frac{n}{n-2}}}\,
  \Big\| e^{1+|v|^2} \Big\|^{\frac{2-\gamma (n-2)}{2}}_{L^{1}} \leq 
$$
\begin{equation*}
\mu c^{(1)}_n \varkappa \Big\| e^{1+|v|^2} \Big\|_{\mathbf{L}^{\frac{n}{n-2}}} + 
\Big(\frac{2}{ \gamma e }\Big)^{\frac{4}{2-\gamma n}}
\Big(\frac{\gamma n}{2 \varkappa \mu c^{(1)}_n }\Big)^{\frac{\gamma n}{2-\gamma n}}
\frac{(2-\gamma n)}{2}
\Big\| e^{1+|v|^2} \Big\|^{\frac{2-\gamma (n-2)}{2-\gamma n}}_{L^{1}}
\end{equation*}
for any $v \in B^{k,2s,s}_\mathrm{vel,0} (I) $, 
$\gamma \in (0, \frac{2}{n})$ and $\varkappa>0$, 
 where the last bound follows from 
Young inequality  \eqref{eq.Young}   
with $p_1=\frac{2}{\gamma n}$, $p_2 = \frac{2}{2-\gamma n}$. 
Then 
\begin{equation} \label{eq.u.gamma.exp}
 \mu^{-1}C_{{\mathcal M},n}  \int_0^t  \Big\|(1+|v|^{2})^2 e^{1+|v|^2 }\Big\|_{L^1}
d\tau \leq 
\end{equation}
\begin{equation*} 
 \int_0^{t} \Big(\mu c^{(1)}_n  
\Big\| e^{1+|v|^2} \Big\|_{\mathbf{L}^{\frac{n}{n-2}}} +
\frac{c_n}{ \mu ^{\frac{2+\gamma n}{2-\gamma n}}\gamma ^{ \frac{4-n \gamma }{2-n\gamma}}}\
 \Big\| e^{1+|v|^2} \Big\|^{\frac{2-\gamma (n-2)}{2-\gamma n}}_{\mathbf{L}^{1}}\Big) d\tau
\end{equation*}
for any $v \in B^{k,2s,s}_\mathrm{vel,0} (I) $ and 
$\gamma \in (0,\frac{1}{n}]$ with a constant $c_n$ 
independent on $\gamma$ and $v$.

Now, combining \eqref{eq.exp.int.im.not} with \eqref{eq.u.gamma.exp} 
for $\gamma =\frac{2}{4m+n}$ we see that
\begin{equation} \label{eq.exp.int.im.not.gamma}
\frac{m}{\mu \, \ln{(m)}}    \int_0^{t_m} 	\left\|
 e^{ 1+|u_{i_m}|^2  }\right\|^{\frac{m+1}{m}}_{L^{1}   } 
d\tau <
\frac{c_n}{ \mu ^{\frac{2m+n}{2m}}} \Big(\frac{4m+n}{2} \Big)^{\frac{8m+n}{4m}} \int_0^{t_m} 
 \Big\| e^{1+|u_{i_m}|^2} \Big\|^{\frac{2m+1}{2m}}_{\mathbf{L}^{1}}
d\tau 
\end{equation}
for all $m \geq 1/\delta_0>n$. 

On the other hand, by the H\"older inequality, for all $m\geq 1$, 
$$
\left\|
 e^{ 1+|u_{i_m}|^2  }\right\|
_{L^{\frac{2m+1}{2m}} 
([0,t_m], L^{1})   }
\leq \left\|
 e^{ 1+|u_{i_m}|^2  }\right\| _{L^{\frac{m+1}{m}} 
([0,t_m], L^{1})   } T^{\frac{m}{(m+1)(2m+1)}}
$$
and  therefore using \eqref{eq.asymp.1} and \eqref{eq.exp.int.im.not.gamma} 
we arrive at the following:
\begin{equation*}
1\leq \frac{ c_n   \ln{(m)}}{
\mu ^{\frac{n}{2m} } m} \Big(\frac{4m+n}{2} \Big)^{\frac{8m+n}{4m}}
T^{\frac{1}{2(m+1)}} 
\left\| e^{ 1+|u_{i_m}|^2  }\right\| ^{-\frac{1}{2m}}_{L^{\frac{m+1}{m}} 
([0,t_m], L^{1})   } \leq 
\end{equation*}
\begin{equation}\label{eq.asymp.2}
\frac{ c_n  \ln{(m)}}{\mu ^{\frac{n}{2m}}  m^3}
 \Big(\frac{4m+n}{2} \Big)^{\frac{8m+n}{4m}}
T^{\frac{1}{2(m+1)}}  \mbox{ for all } m \geq 1/\delta_0.
\end{equation}
Again, as $\frac{8m+n}{4m}<5/2$ for all $ m \geq 1/\delta_0 >n$, 
passing to the limit with respect to $m\to +\infty$ in \eqref{eq.asymp.2} we obtain 
a contradiction:
$
1\leq  0,
$ 
i.e. the lemma is proved.
\end{proof}

Next, formula \eqref{eq.exp.int.i} and Lemma \ref{l.claim} 
imply that there is a number $\delta$, $0<\delta\leq \delta_0<1/n$, 
satisfying \eqref{eq.Gron.add.im} and such that  
\begin{equation*}
    \left\| e^{ 1+|u_i|^2  } (\cdot,t)\right\|_{L^{1}   }
\leq (A+1)+\int_0^t \!\! \Big(  B 
 \left\| e^{ 1+|u_i|^2  }\right\|_{L^{1}   }
 +  \frac{1}{\mu \, \delta\, \ln{(1/\delta)}} \left\|
 e^{ 1+|u_{i}|^2  }\right\|^{1+\delta}_{L^{1}   } 
\Big) d\tau 
\end{equation*}
for all $i \in {\mathbb N}$ and all $t \in [0,T]$. 
Actually, inequality \eqref{eq.Gron.add.im} provides additional 
bound \eqref{eq.Gron.add} 
for $\gamma_0 = 1+\delta>1$, ${\mathfrak A} = A+1$, 
${\mathfrak B} (t)= B$,  
${\mathfrak C}(t) = \frac{1}{\mu \, \delta \,  \ln{(1/\delta)}}$ and $h=T$ in Perov  
Lemma \ref{l.Perov}. Therefore this lemma 
yields for all $i \in {\mathbb N}$ and all $t \in [0,T]$: 
\begin{equation} \label{eq.Gron.large}
 \left\| e^{ 1+|u_i|^2  } (\cdot,t)\right\|_{L^{1}   } \leq (A+1) \,  \left( 
e^{-BT\delta} -
\frac{ (A+1)^{\delta} }{\mu \, \ln{(1/\delta)}} 
 \int_0^t e^{(\tau-t)B\delta}d \tau 
\right)^{-1/\delta} 
\end{equation} 
with the number $\delta=\delta(f,u_0)$, $0<\delta\leq \delta_0<1/n$, granted by 
Lemma \ref{l.claim}. 

In particular, the sequence $\{  e^{ 1+|u_i|^2  } \}$ is bounded in 
$C(I, L^1)$. Therefore using \eqref{eq.s=r} (in particular,  the compactness 
of the set ${\mathbb T}^n \times I$) we conclude that 
$$
\|e^{ 1+|u_i|^2}  \|_{C(I, L^1)} \geq 
T^{-1} \|e^{ 1+|u_i|^2}  \|_{L^1 (I, L^1)} 
\geq T^{-1} C_{n,T} \|u_i\| ^{2(n+2)} _{L^{n+2}(I,\mathbf{L}^{n+2}) }
$$
i..e.  $\{  u_i  \}$ is bounded in 
$L^{n+2}(I, \mathbf{L}^{n+2})$. 
Finally, the statement of the theorem follows from Lemma \ref{l.closure} 
since the pair $\mathfrak{s}=n+2$,  $\mathfrak{r}=n+2$
satisfies \eqref{eq.s.r}. 
\end{proof}
 
\begin{remark} \label{r.implicit}
Note that the constants $A=A(u_0)$, $B=B(f)$ 
in estimate \eqref{eq.Gron.large} depends rather 
directly on the data $(f,u_0)$ (see \eqref{eq.A} and \eqref{eq.B}) while the number 
$\delta = \delta(f,u_0)$ 
is given somehow implicitly via Lemma \ref{l.claim} and it depends actually on the 
corresponding sequence $\{u_i\}$ from the inverse image of mapping \eqref{eq.NS.map}. 
Thus, for $n\geq 3$, this leaves us basic Energy Estimate 
\eqref{eq.En.Est2imp} as the only one known 
universal a priori estimate for the Navier-Stokes type equations \eqref{eq.NS} with 
non-linearity \eqref{eq.non-linearity} satisfying \eqref{eq.trilinear} even in the 
`local'{} case $a=0$.  Moreover, taking into account the type of the  obtained 
implicit estimate \eqref{eq.Gron.large}, 
questions arise on  the credibility of the numerical simulations involving 
regular/smooth solutions to  equations this type. 
\end{remark}

Next, given a Fr\'{e}chet space $\mathcal{F}$, we denote by $C^\infty (I,{\mathcal F})$ the
space of all infinitely differentiable functions of $t \in I=[0,T]$ with values in 
$\mathcal{F}$.

\begin{corollary}
\label{c.exist.smooth}
Let non-linearity \eqref{eq.non-linearity} satisfy \eqref{eq.trilinear}.
If $n=2$ or $a=0$ then \eqref{eq.NS} induce a homeomorphism 
of   $\mathbf{C}^\infty ({\mathbb T}^n \times I)$ onto the space
   $\mathbf{C}^\infty ({\mathbb T}^n \times I) \times \mathbf{C}^{\infty} $.
\end{corollary}

\begin{proof}
It follows immediately from
   the Sobolev embedding theorem  that 
$$\mathbf{C}^\infty=\cap_{s=0}^\infty \mathbf{H}^{2s}, \, 
  \mathbf{C}^\infty ({\mathbb T}^n \times I) = C^\infty (I, \mathbf{C}^\infty)
  =   \cap_{s=1}^\infty B^{0,2s,s}_\mathrm{vel,0} (I) = 
\cap_{s=0}^\infty B^{0,2s,s}_\mathrm{for} (I).
$$
Then an application of Proposition \ref{p.exist.Bohcher.n=2} and
 Theorem \ref{t.exist.Bohcher} finishes the proof. 
\end{proof}

\begin{example}  \label{ex.SvPl} 
If we consider $a=0$ and 
non-linearity \eqref{eq.SvPl} from \cite{PlSv03} then 
for all $b \in (0,1)$ and $u \in {\mathcal H}^{(0)}_{2s+k} = \mathbf{H}^{2s+k}$, 
$2s+k >\frac{n}{2}$,  we have 
\begin{equation} \label{eq.NS.zero}
((u \cdot \nabla) u , u)_{\mathbf{L}^2} = - 
(u, (u \cdot \nabla) u )_{\mathbf{L}^2} - ((\mathrm{div } u) u,u) _{\mathbf{L}^2}, 
\end{equation}
$$
(b \, (u \cdot \nabla) u + 
\frac{1}{2} (1-b) \, \nabla |u|^2 + 
\frac{1}{2} (\mathrm{div } u) u, u)_{\mathbf{L}^2} = 
$$
$$
\int_{\mathbb{T}^n}
\Big( - \frac{b}{2} \, (\mathrm{div } u) |u|^2 -\frac{1}{2} (1-b) \, (\mathrm{div } u) |u|^2
+ \frac{1}{2} \, (\mathrm{div } u) |u|^2 \Big) =0.
$$ 
Thus, non-linearity \eqref{eq.SvPl} satisfies \eqref{eq.trilinear} for all $n\geq 2$ in 
this case. Hence it follows from Proposition \ref{p.exist.Bohcher.n=2} and 
Theorem \ref{t.exist.Bohcher}  that  initial problem  \eqref{eq.NS} with $a=0$  
and non-linearity \eqref{eq.SvPl} 
has unique solution $u \in B^{k,2s,s}_\mathrm{vel,0} (I)$ for any 
data $(f,u_0) $ from the space $ B^{k,2(s-1),s-1}_\mathrm{for} (I)
\times \mathbf{H}^{2s+k}$ if $n \geq 3$, $2s+k>n/2+2$, or 
$n=2$, $s\in {\mathbb N}$, $k \in {\mathbb Z}_+$.  
Corollary   \ref{c.exist.smooth}  implies that it 
admits unique smooth solutions for any smooth data.
\end{example}

\begin{example} \label{ex.NS}
If $a=1$ and ${\mathbf D} u =  (u \cdot \nabla) u $ 
then according to Example \ref{ex.NS.intr} relations \eqref{eq.NS} give us  
the Navier-Stokes equations for incompressible fluid.  
In particular, ${\mathcal H}^{(1)}_{2s+k} = V_{2s+k}$ and  
\eqref{eq.NS.zero} implies that ${\mathbf D} u$ satisfies 
\eqref{eq.trilinear} for all $n\geq 2$  in this case, see \cite{Tema79,Tema95}. Since ${\mathbf D} u$ satisfies 
\eqref{eq.trilinear} then Proposition \ref{p.exist.Bohcher.n=2} 
readily implies the well-known result that Navier-Stokes equations \eqref{eq.NS}  
has unique regular (smooth) solution $(u,p)$ for any regular (smooth) data $(f,u_0)$ if 
$n=2$, cf. works \cite{Lady70,Lady03} 
or \cite{Serr59b}, \cite{Tema95} for the periodic setting. 
But for $n\geq 3$ inequalities \eqref{eq.s.r} and \eqref{eq.sigma.rho}
never coincide. For $a=1$ the arguments from the proof of Theorem \ref{t.exist.Bohcher} 
fail, too,    because we have to estimate the pseudo-differential term 
$\mathbf{P} \mathbf{D} u$ instead of the 
\textit{local} term $\mathbf{D} u$, see \eqref{eq.non-lin.term}, and the integral operator 
$\mathbf{P}$ does not admit point-wise estimates. 
\end{example}

\smallskip

\textit{Funding.\,}
The research was supported by the Russian Science Foundation,  grant N 20-11-20117.

\textit{Conflicts of interest/Competing interests.\,} No such matters noticed

\textit{Availability of data and material.\,} Not applicable

\textit{Code availability.\,} Not applicable


\begin{thebibliography}{XXXXXX}

\bibitem{Ad03}
Adams, R.,
  \textit{Sobolev Spaces},
  Pure and Applied Mathematics, V. 140, Academic Press, 2003.

\bibitem{Agra90}
Agranovich, M.~S.,
  \textit{Elliptic operators on closed manifold},
  In: Current Problems of Mathematics, Fundamental Directions, Vol. 63,
      VINITI, 1990, 5--129.
			
\bibitem{deRh55} G. De Rham, 
{\it Vari\'et\'es Diff\'erentiables}, Hermann$\&$C, \'Editeurs, Paris, 1955. 

\bibitem{ESS03}
  Escauriaza, L.,  Seregin, G.~A.,  and $\check{\textrm{S}}$verak, V.,
  \textit{$L^{3,\infty}$-solutions of the Navier-Stokes equations and backward uniqueness},
  Russian Mathematical Surveys \textbf{58} (2003), no.~2, 211--250.

\bibitem{FursVish80}
Fursikov, A.~V., and Vishik, M.~I.,
  \textit{Mathematical Problems of Statistical Hydromechanics},
  Nauka, Moscow, 1980, 440~pp.

\bibitem{Ham82}
Hamilton, R. S.,
  \textit{The inverse function theorem of Nash and Moser},
  Bull. of the AMS \textbf{7} (1982), no.~1, 65--222.

\bibitem{Hopf51}
Hopf, E.,
  \textit{\"{U}ber die Anfangswertaufgabe f\"{u}r die hydrodynamischen Grundgleichungen},
  Math. Nachr. \textbf{4} (1951), 213--231.

\bibitem{Kolm42}
Kolmogorov, A.~N.,
  \textit{Equations of turbulent mouvement of incompressible fluid},
  Izv. AN SSSR, Physics Series \textbf{6} (1942), no.~1, 56--58.

\bibitem{Lady67}
Ladyzhenskaya, O.~A.,
  \textit{On the uniqueness and on the smoothness of weak solutions of the Navier-Stokes equations},
  Zap. Nauchn. Sem. LOMI \textbf{5} (1967), 169--185.

\bibitem{Lady70}
Ladyzhenskaya, O.~A.,
  \textit{Mathematical Problems of Incompressible Viscous Fluid},
  Nauka, Moscow, 1970, 288~pp.

\bibitem{Lady03}
Ladyzhenskaya, O.~A.,
  \textit{The sixth prize millenium problem: Navier-Stokes equations, existence and 
	smoothness},  Russian Math. Surveys \textbf{58} (2003), no.~2, 251--286.

\bibitem{LadSoUr67}
Ladyzhenskaya,~O. A., Solonnikov, V. A., and Ural'tseva, N. N.,
  \textit{Linear and Quasilinear Equations of Parabolic Type},
  Nauka, Moscow, 1967.

\bibitem{Lera34a}
Leray, J.,
 \textit{Essai sur les mouvements plans d'un liquid visqueux que limitend des parois},
 J. Math. Pures Appl. \textbf{9} (1934), 331--418.

\bibitem{Lera34b}
Leray, J.,
  \textit{Sur le mouvement plans d'un liquid visqueux emplissant l'espace},
  Acta Math. \textbf{63} (1934), 193--248.

\bibitem{Lion61}
Lions, J.-L.,
  \textit{\'{E}quations diff\'{e}rentielles op\'{e}rationelles et probl\`{e}mes aux limites},
  Sprin\-ger-Verlag, Berlin, 1961.

\bibitem{Lion69}
Lions, J.-L.,
  \textit{Quelques m\'{e}thodes de r\'{e}solution des probl\`{e}mes aux limites non lin\'{e}are},
  Dunod/Gauthier-Villars, Paris, 1969, 588~pp.

\bibitem{MPF91}
Mitrinovi\'c, D.~S., Pe$\check{c}$ari\'c, J.~E, and Fink, A.~M.,
  \textit{Inequalities Involving Functions and Their Integrals and Derivatives},
  Mathematics and its Applications (East European Series), V. 53,
  Kluwer Academic Publishers, Dordrecht, 1991.

\bibitem{Nir59}
Nirenberg, L.,
  \textit{On Elliptic partial differential equations},
  Ann. Sc. Norm. Sup. di Pisa, Cl. Sci., Ser. 3., \textbf{13} (1959), 115--162.

\bibitem{Per59}
Perov, A.~I.,
  \textit{K voprosu o strukture integral'noi voronki},
  Nauch. Dokl. Vyssh. Shkoly (1959), no.~2, 60--66.
	
\bibitem{PlSv03}
Plech\'a$\rm \check{c}$, P., $\rm \check{S}$ver\'ak, V., 
\textit{Singular and regular
solutions of a nonlinear parabolic system}, Nonlinearity 16 (2003), no. 6, 2083--2097.

\bibitem{Po22}
Polkovnikov, A., 
\textit{An open mapping theorem for nonlinear operator equations associated with elliptic 
complexes}, Applicable Analysis, 2022, https://doi.org/10.1080/
00036811.2021.2021190

\bibitem{Pro59}
Prodi, G.,
  \textit{Un teorema di unicit\'a per le equazioni di Navier-Stokes},
  Annali di Matematica Pura ed Applicata (1959), no. ~48, 173--182.

\bibitem{Serr59b}
Serrin, J.,
  \textit{Mathematical Principles of Classical Fluid Mechanics},
  Encyclopedia of Physics, Springer-Verlag, 1959.

\bibitem{Serr62}
Serrin, J.,
  \textit{On the interior regularity of weak solutions of the Navie-Stokes equations},
  Archive for Rational Mechanics and Analysis \textbf{9} (1962), 187--195.

\bibitem{ShlTa21}
Shlapunov, A.A., Tarkhanov, N., 
\textit{An open mapping theorem for the Navier-Stokes type equations associated with the de 
Rham complex over ${\mathbb R}^n$}, Siberian Electronic Math. Reports, 18:2 (2021), 1433-1466.

\bibitem{ShlTaArxiv}
Shlapunov, A.A., Tarkhanov, N., 
\textit{Inverse image of precompact sets and existence theorems for the Navier-Stokes 
equations in spatially periodic setting}, arxiv.org/abs/2106.07515.

\bibitem{Sm65}
Smale, S.,
  \textit{An infinite dimensional version of Sard's theorem},
  Amer. J. Math. \textbf{87} (1965), no.~4, 861--866.

\bibitem{Tema79}
Temam, R.,
  \textit{Navier-Stokes Equations. Theory and Numerical Analysis},
  North Holland Publ. Comp., Amsterdam, 1979.

\bibitem{Tema95}
Temam, R.,
  \textit{Navier-Stokes Equations and Nonlinear Functional Analysis},
  2\,nd ed., SIAM, Philadelphia, 1995.

\bibitem{Tao16}
Tao, T.,
  \textit{Finite time blow-up for an averaged three-dimensional Navier-Stokes equation},
  J. of the AMS \textbf{29} (2016), 601--674.

\end{thebibliography}
 \end{document}